\documentclass[12pt]{amsart}
\usepackage{amsmath}
\usepackage{amsxtra}
\usepackage{enumitem}
\usepackage{amscd}
\usepackage{tikz-cd}
\usepackage{amsthm}
\usepackage{hyperref}
\usepackage[T1]{fontenc}
\usepackage{amsfonts}
\usepackage{amssymb}
\usepackage{eucal}
\usepackage[all]{xypic}
\usepackage{color} 
\newtheorem{Theo}[subsubsection]{Theorem}
\newtheorem{Theor}{Theorem}
\newtheorem{Theore}{Theorem}
\newtheorem*{Theorem}{Theorem}

\newtheorem{cor}[subsubsection]{Corollary}

\newtheorem{Conj}{Conjecture}
\newtheorem{Theorema}{Theorem}

\theoremstyle{definition}

\newtheorem{conj}[subsection]{Conjecture}
\newtheorem{Conje}{Conjecture}
\newtheorem{Coro}{Corollary}
\newtheorem{Corol}{Corollary}

\newtheorem{fact}[subsubsection]{Fact} 
\newtheorem{rem}[subsubsection]{Remark}

\newtheorem{question}[subsection]{Question}
\newtheorem{Prop}[subsubsection]{Proposition}
\newtheorem{Obs}[subsubsection]{Observation}
\newtheorem{example}[subsubsection]{Example}
\newtheorem{Convention}[subsubsection]{Convention}
\newtheorem{Lemm}[subsubsection]{Lemma} 

\newtheorem{definition}[subsubsection]{Definition}
\newtheorem{problem}[subsection]{Problem}

\newtheorem{ex}[subsection]{Exercise}

\usepackage{yhmath}
\DeclareSymbolFont{largesymbols}{OMX}{yhex}{m}{n}
\DeclareMathAccent{\widetilde}{\mathord}{largesymbols}{"65}

\newcommand{\bp}{\begin{Prop}}
\newcommand{\ep}{\end{Prop}}

\newcommand{\bl}{\begin{Lemm}}
\newcommand{\el}{\end{Lemm}}
\newcommand{\bex}{\begin{ex} \rm}
\newcommand{\eex}{\end{ex}}
\newcommand{\bt}{\begin{Theo}}
\newcommand{\et}{\end{Theo}}
\newcommand{\bq}{\begin{question}}
\newcommand{\eq}{\end{question}}

\newcommand{\bc}{\begin{cor}}
\newcommand{\ec}{\end{cor}}
\newcommand{\bob}{\begin{Obs}}
\newcommand{\eob}{\end{Obs}}

\newcommand{\nc}{\newcommand}
\nc{\renc}{\renewcommand}
\nc{\ssec}{\subsection}
\nc{\sssec}{\subsubsection} 
\nc\ol{\overline}
\nc\wt{\widetilde}
\nc\wh{\widehat}
\nc\tboxtimes{\wt{\boxtimes}}

\renc{\d}{{\delta}}
\nc{\Aa}{{\mathbb{A}}}
\nc{\Bb}{{\mathbb{B}}}
 \nc{\Gg}{{\mathbb{G}}}  
\nc{\Hh}{{\mathbb{H}}}
 \nc{\Nn}{{\mathbb{N}}}
\nc{\Pp}{{\mathbb{P}}}
\nc{\Rr}{{\mathbb{R}}}
\newcommand{\F}{\mathbb{F}}
\nc{\BV}{{\mathbb{V}}}
\nc{\BW}{{\mathbb{W}}}
\newcommand{\Z}{\mathbb{Z}}
\newcommand{\N}{\mathbb{N}}

\nc{\Qq}{{\mathbb{Q}}}
\nc{\Ss}{{\mathbb{S}}}
\nc{\Cc}{{\mathbb{C}}}
\nc{\Ff}{{\mathbb{F}}}
 \nc{\EL}{{L_\infty}}

\nc{\CA}{{\mathcal{A}}}
\nc{\CB}{{\mathcal{B}}}

\nc{\CE}{{\mathcal{E}}}
\nc{\CF}{{\mathcal{F}}}

\nc{\Las}{\mathsf{Las}}
\nc{\CG}{{\mathcal{G}}}

\nc{\CL}{{\mathcal{L}}}
\nc{\CC}{{\mathcal{C}}}
\nc{\CM}{{\mathcal{M}}}

\nc{\CN}{{\mathcal{N}}}
\nc{\Oog}{{\mathbb{O}}}
\nc{\Oo}{{\mathcal{O}}}
\nc{\CP}{{\mathcal{P}}}
\nc{\CQ}{{\mathcal{Q}}}
\nc{\CR}{{\mathcal{R}}}
\nc{\CS}{{\mathcal{S}}}
\nc{\CT}{{\mathcal{T}}}
\nc{\CU}{{\mathcal{P}}}

\nc{\CV}{{\mathcal{V}}}
 
\nc{\CW}{{\mathcal{W}}}
\nc{\CZ}{{\mathcal{Z}}}

\nc{\cM}{{\check{\mathcal M}}{}}
\nc{\csM}{{\check{\mathcal A}}{}}
\nc{\oM}{{\overset{\circ}{\mathcal M}}{}}
\nc{\obM}{{\overset{\circ}{\mathbf M}}{}}
\nc{\oCA}{{\overset{\circ}{\mathcal A}}{}}
\nc{\obA}{{\overset{\circ}{\mathbf A}}{}}
\nc{\ooM}{{\overset{\circ}{M}}{}}
\nc{\osM}{{\overset{\circ}{\mathsf M}}{}}
\nc{\vM}{{\overset{\bullet}{\mathcal M}}{}}
\nc{\nM}{{\underset{\bullet}{\mathcal M}}{}}
\nc{\oD}{{\overset{\circ}{\mathcal D}}{}}
\nc{\obD}{{\overset{\circ}{\mathbf D}}{}}
\nc{\oA}{{\overset{\circ}{\mathbb A}}{}}
\nc{\op}{{\overset{\bullet}{\mathbf p}}{}}
\nc{\cp}{{\overset{\circ}{\mathbf p}}{}}
\nc{\oU}{{\overset{\bullet}{\mathcal U}}{}}
\nc{\oZ}{{\overset{\circ}{\mathcal Z}}{}}
\nc{\ofZ}{{\overset{\circ}{\mathfrak Z}}{}}
\nc{\oF}{{\overset{\circ}{\fF}}}

\nc{\fa}{{\mathfrak{a}}}
\nc{\fb}{{\mathfrak{b}}}
\nc{\fg}{{\mathfrak{g}}}
\nc{\fgt}{{\fg}_!}
\nc{\fgl}{{\mathfrak{gl}}}
\nc{\fh}{{\mathfrak{h}}}
\nc{\fj}{{\mathfrak{j}}}
\nc{\fm}{{\mathfrak{m}}}
\nc{\ft}{{\mathfrak{t}}}
\nc{\fn}{{\mathfrak{n}}}
\nc{\fu}{{\mathfrak{u}}}
\nc{\fp}{{\mathfrak{p}}}
\nc{\fr}{{\mathfrak{r}}}
\nc{\fs}{{\mathfrak{s}}}
\nc{\fsl}{{\mathfrak{sl}}}
\nc{\hsl}{{\widehat{\mathfrak{sl}}}}
\nc{\hgl}{{\widehat{\mathfrak{gl}}}}
\nc{\hg}{{\widehat{\mathfrak{g}}}}
\nc{\chg}{{\widehat{\mathfrak{g}}}{}^\vee}
\nc{\hn}{{\widehat{\mathfrak{n}}}}
\nc{\chn}{{\widehat{\mathfrak{n}}}{}^\vee}

\nc{\fA}{{\mathfrak{A}}}
\nc{\fB}{{\mathfrak{B}}}
\nc{\fD}{{\mathfrak{D}}}
\nc{\fE}{{\mathfrak{E}}}
\nc{\fF}{{\mathfrak{F}}}
\nc{\fG}{{\mathfrak{G}}}
\nc{\fK}{{\mathfrak{K}}}
\nc{\fL}{{\mathfrak{L}}}
\nc{\fM}{{\mathfrak{M}}}
\nc{\fN}{{\mathfrak{N}}}
\nc{\fP}{{\mathfrak{P}}}
\nc{\fU}{{\mathfrak{U}}}
\nc{\fV}{{\mathfrak{V}}}
\nc{\fZ}{{\mathfrak{Z}}}
\newcommand{\Q}{\mathbb{Q}}
\nc{\bb}{{\mathbf{b}}}
\nc{\bd}{\partial}
\nc{\be}{{\mathbf{e}}}
\nc{\bj}{{\mathbf{j}}}
\nc{\bn}{{\mathbf{n}}}
\nc{\bF}{{\mathbf{F}}}
\nc{\bu}{{\mathbf{u}}}
\nc{\bv}{{\mathbf{v}}}
\nc{\bx}{{\mathbf{x}}}
\nc{\bs}{{\mathbf{s}}}
\nc{\by}{{\bar{y}}}
\nc{\bw}{{\mathbf{w}}}
\nc{\bA}{{\mathbf{A}}}
\nc{\bK}{{\mathbf{K}}}
\nc{\bI}{{\mathbf{I}}}
\nc{\bB}{{\mathbf{B}}}
\nc{\bG}{{\mathbf{G}}}

\nc{\bD}{{\mathbf{D}}}
\nc{\bP}{{\mathbf{P}}}
\nc{\bH}{{\mathbf{H}}}
\nc{\bM}{{\mathbf{M}}}
\nc{\bN}{{\mathbf{N}}}
\nc{\bV}{{\mathbf{V}}}
\nc{\bU}{{\mathbf{U}}}
\nc{\bL}{{\mathbf{L}}}

\nc{\bW}{{\mathbf{W}}}
\nc{\bX}{{\mathbf{X}}}
\nc{\bY}{{\mathbf{Y}}}
\nc{\bZ}{{\mathbf{Z}}}
\nc{\bS}{{\mathbf{S}}}
\nc{\bSi}{{\bar{\Sigma}}}
\nc{\sA}{{\mathsf{A}}}
\nc{\sB}{{\mathsf{B}}}
\nc{\sC}{{\mathsf{C}}}
\nc{\sD}{{\mathsf{D}}}
\nc{\sF}{{\mathsf{F}}}
\nc{\sG}{{\mathsf{G}}}
\nc{\sK}{{\mathsf{K}}}
\nc{\sM}{{\mathsf{M}}}
\nc{\sO}{{\mathsf{O}}}
\nc{\sQ}{{\mathsf{Q}}}
\nc{\sP}{{\mathsf{P}}}
\nc{\sZ}{{\mathsf{Z}}}
\nc{\sfp}{{\mathsf{p}}}
\nc{\sr}{{\mathsf{r}}}
\nc{\sg}{{\mathsf{g}}}
\nc{\sff}{{\mathsf{f}}}
\nc{\sfb}{{\mathsf{b}}}
\nc{\sfc}{{\mathsf{c}}}
\nc{\sd}{{\ltimes}} 

\nc{\tH}{{\widetilde{H}}}
\nc{\tA}{{\widetilde{\mathbf{A}}}}
\nc{\tB}{{\widetilde{\mathcal{B}}}}
\nc{\tg}{{\widetilde{\mathfrak{g}}}}
\nc{\tG}{{\widetilde{G}}}

\nc{\TM}{{\widetilde{\mathbb{M}}}{}}
\nc{\tO}{{\widetilde{\mathsf{O}}}{}}
\nc{\tU}{\widetilde{U}}
\nc{\TZ}{{\tilde{Z}}}
\nc{\tx}{{\tilde{x}}}
\nc{\tq}{{\tilde{q}}}

\nc{\tfP}{{\widetilde{\mathfrak{P}}}{}}
\nc{\tz}{{\tilde{\zeta}}}
\nc{\tmu}{{\tilde{\mu}}}

  \nc{\vol}{{\mathop{\operatorname{\rm vol\,}}}}
  
  \nc{\gal}{{\mathop{\operatorname{\rm Gal\,}}}}
  \nc{\cl}{{\mathop{\operatorname{\rm cl}}}}
  \nc{\disc}{{\mathop{\operatorname{\rm disc}}}}
  \nc{\Sym}{{\mathop{\operatorname{\rm Sym}}}}
   \nc{\Aut}{{\mathop{\operatorname{\rm Aut}}}}
 \nc{\Spec}{{\mathop{\operatorname{\rm Spec}}}}
  \nc{\spec}{{\mathop{\operatorname{\rm Spec}}}}
\nc{\Ker}{{\mathop{\operatorname{\rm Ker}}}}
 \nc{\dom}{{\mathop{\operatorname{\rm dom}}}}
\nc{\End}{{\mathop{\operatorname{\rm End}}}}
 \nc{\Hom}{\operatorname{\Hom}}
 \nc{\GL}{{\mathop{\operatorname{\rm GL}}}}
 \nc{\Id}{{\mathop{\operatorname{\rm Id}}}}
 \nc{\rk}{{\mathop{\operatorname{\rm rk}}}}
 \nc{\length}{{\mathop{\operatorname{\rm length}}}}
\nc{\supp}{{\mathop{\operatorname{\rm supp} \, }}}
\nc{\val}{{\rm val}}
\nc{\res}{{\mathop{\operatorname{\rm res}}}}

\def\Ind#1#2#3{{#1} {\downarrow}_{#3} {#2} }

\nc{\seq}[1]{\stackrel{#1}{\sim}}

\def\beq#1{\begin{equation} \label{ #1}}
\def\eeq{\end{equation}}

\def\prf{\begin{proof}}
\def\pv{\end{proof} }
 \def\eprf{\end{proof} }

 \renc{\b}{{\beta}}

\def\Ind#1#2{#1\setbox0=\hbox{$#1x$}\kern\wd0\hbox to 0pt{\hss$#1\mid$\hss}
\lower.9\ht0\hbox to 0pt{\hss$#1\smile$\hss}\kern\wd0}

\usepackage{color}
\usepackage{hyperref}

\usepackage{url}

 \makeatletter
\def\@setthanks{\vspace{-\baselineskip}\def\thanks##1{\@par##1\@addpunct.}\thankses}
\makeatother
\title{Diophantine problems over tamely ramified fields}
\author{Konstantinos Kartas}
\thanks{During this research, the author was funded by EPSRC grant EP/20998761 and was also supported by the Onassis Foundation - Scholarship ID: F ZP 020-1/2019-2020.}

\newcommand{\Addresses}{{
  \bigskip
  \footnotesize

\textsc{Mathematical Institute, Woodstock Road, Oxford OX2 6GG.}\par\nopagebreak
  \textit{E-mail address}: \texttt{kartas@maths.ox.ac.uk}
}}

\begin{document}


\maketitle
\begin{abstract}
Assuming a certain form of resolution of singularities, we prove a general \textit{existential} Ax-Kochen/Ershov principle for tamely ramified fields in all characteristics. This specializes to well-known results in residue characteristic $0$ and unramified mixed characteristic. It also encompasses the conditional existential decidability results known for $\F_p(\!(t)\!)$ and its finite extensions, due to Denef-Schoutens. On the other hand, it also applies to the setting of \textit{infinite} ramification, thereby providing us with an abundance of existentially decidable infinitely ramified extensions of $\Q_p$ and $\F_p(\!(t)\!)$. 


\end{abstract}
\setcounter{tocdepth}{1}
\tableofcontents

\section*{Introduction} 

The decidability of the $p$-adic numbers $\Q_p$, established by Ax-Kochen \cite{AK} and Ershov \cite{Ershov}, still remains one of the highlights of model theory. It motivated several decidability results both in mixed and positive characteristic: 
\begin{itemize}
\item In mixed characteristic, Kochen \cite{Kochen} showed that $\Q_p^{ur}$, the maximal unramified extension of $\Q_p$, is decidable. More generally, by work of \cite{Ziegler}, \cite{Ershov}, \cite{Bas2}, \cite{Bel} and more recently \cite{AJ}, \cite{Junguk} and \cite{Lee}, we have a good understanding of the model theory of unramified and finitely ramified mixed characteristic henselian fields. 
\item In positive characteristic, our understanding is much more limited. Nevertheless, by work of Denef-Schoutens \cite{Den}, we know that $\F_p(\!(t)\!)$ is \textit{existentially} decidable in $L_t=\{0,1,t,+,\cdot\}$, modulo resolution of singularities. In fact, Theorem 4.3 \cite{Den} applies to show that any finitely ramified extension of $\F_p(\!(t)\!)$ is existentially decidable relative to its residue field.
\end{itemize}
The situation is less clear for \textit{infinitely} ramified extensions of $\Q_p$ and $\F_p(\!(t)\!)$. Macintyre discusses two such interesting \textit{wildly ramified} extensions on pg.140 \cite{Mac}, namely $\Q_p(\zeta_{p^{\infty}})$ and $\Q_p^{ab}$. The author has shown in \cite{KK} that these fields are (existentially) decidable relative to the perfect hulls of $\F_p(\!(t)\!)$ and $\overline{\F}_p(\!(t)\!)$ respectively in the language $L_t$. Another interesting example is $\Q_p(p^{1/p^{\infty}})$, which is also (existentially) decidable relative to the perfect hull of $\F_p(\!(t)\!)$ in $L_t$. However, the main results of \cite{KK} do not say anything about whether any of these fields is actually decidable or even existentially decidable. 

In the present paper, we shift our attention from the wildly ramified setting to the tamely ramified setting. We address the case of \textit{tamely} ramified extensions of $\Q_p$ and $\F_p(\!(t)\!)$, modulo resolution of singularities, including also the \textit{infinitely} ramified ones. In particular, we obtain an abundance of explicit examples of \textit{infinitely} ramified extensions of $\Q_p$ and $\F_p(\!(t)\!)$ whose theory is existentially decidable. Before we state our results, let us first introduce the precise form of resolution of singularities that we will use:
\begin{Conj} [Log-Resolution] \label{R}
Let $X$ be a reduced, flat scheme of finite type over an excellent discrete valuation ring $R$. Then there exists a blow-up morphism $f:\tilde{X}\to X$ in a nowhere dense center $Z\subsetneq X$ such that 
\begin{enumerate}
\item $\tilde{X}$ is a regular scheme.
\item $\tilde{X}_s=\tilde{X}\times_{\Spec R} \Spec(R/\mathfrak{m}_R)$ is a strict normal crossings divisor.
\end{enumerate}
\end{Conj}
Some background material related to resolution of singularities is provided in \S \ref{rosection}. We obtain the following existential Ax-Kochen/Ershov principle, from which all applications will be deduced:
\begin{Theor} \label{mainA}
Assume Conjecture \ref{R}. Suppose $(K,v)$ and $(L,w)$ are henselian and tamely ramified over a valued field $(F,v_0)$ with $\Oo_F$ an excellent DVR. If $\text{RV}(K)\equiv_{\exists,RV(F)} \text{RV}(L)$, then $K\equiv_{\exists,F} L$ in $L_{\text{rings}}$.
\end{Theor} 
The RV-structure associated to a valued field is a certain object which integrates the value group and the residue field in a single structure. All necessary background on such structures is provided in \S \ref{rvsec}. The language for RV-structures uses the group language $\{1, \cdot\}$ together with a constant symbol for $\infty$, a ternary predicate for (multi-valued) addition and a binary predicate for the relation $va\leq vb$. We recall this formalism in detail in Definition \ref{rvlangdef}. It should be noted that our notion of a "tamely ramified extension" used in Theorem \ref{mainA} is \textit{not} restricted to algebraic extensions but does indeed specialize to the ordinary notion in the case of an algebraic extension (see \S \ref{transtame}). This level of generality is essential for applications, even if one is merely interested in algebraic tamely ramified algebraic extensions of $\Q_p$ or $\F_p(\!(t)\!)$ (see Remark \ref{transisessential}). The excellence condition on the discrete valuation ring $\Oo_F$ simply amounts to $\widehat{F}/F$ being separable (see \S \ref{rosection}). For example, this is the case for  the valued fields $(\Q(t),v_t)$, $(\Q,v_p)$ or $(\F_p(t),v_t)$.

The key ingredient in the proof of Theorem \ref{mainA} is a certain form of Hensel's Lemma, which we call RV-Hensel's Lemma. This is completely parallel to the classical geometric version of Hensel's Lemma for smooth morphisms over henselian local rings, saying that the existence of an integral point is equivalent to the existence of a solution in the residue field. In general, the best one can hope for is not smoothness but rather strict normal crossings. This necessitates a variant of Hensel's Lemma for morphisms having strict normal crossings. However, to ensure the existence of an integral point, one must now require not just a solution in the residue field but rather an RV-solution, i.e., more residual information is needed. Most crucially, this only works when one of the multiplicities of the irreducible components of the special fiber is not divisible by $p$, the residue characteristic of the base ring. This condition plays well with tame ramification (as one might guess) which is why the scope of our results is restricted to that setting.

Theorem \ref{mainA} specializes to well-known Ax-Kochen/Ershov results in residue characteristic $0$ and in the mixed characteristic \textit{unramified} setting. Moreover, these Ax-Kochen/Ershov principles hold not only for the existential theories but also for the full first-order theories. The case of \textit{finite} tame ramification in mixed characteristic and with perfect residue fields was proved recently in Corollary 5.9 \cite{Junguk} (see also Remark \ref{jungukrem}). 

At the same time, Theorem \ref{mainA} implies a conditional existential decidability result for $\F_p(\!(t)\!)$, which was already known by the work of Denef-Schoutens \cite{Den}. Our proof does not use Greenberg's approximation theorem, which is an important ingredient in \cite{Den}. In the case of $\F_p(\!(t)\!)$, the proof can be simplified significantly, as we explain in \S \ref{Denef-Schoutensec}. Although Conjecture \ref{R} is a more refined version of resolution than Conjecture 1 \cite{Den}, the simpler proof explained in \S \ref{Denef-Schoutensec} uses only Conjecture 1 \cite{Den}. In fact, there is a slight variation of the proof which only depends on \textit{local uniformization}, a local (weaker) form of resolution of singularities. This was stated (and briefly sketched) in the first manuscript version of the present paper (see \texttt{arXiv:2103.14646v1 [math.AG]}). The details are now included in the author's PhD thesis (see \S 5.6 \cite{ThesisKartas}). Recent work by Anscombe-Dittmann-Fehm \cite{AnsDittFehm} has further improved upon this. Their proof depends on an even weaker assumption and uses entirely different techniques from the ones used here or in \cite{Den}.

The main emphasis of the present paper is to go beyond the setting of discretely valued fields and obtain some information about the existential theories of \textit{infinitely} ramified fields. We highlight the following application regarding the maximal tamely ramified extension $\Q_p^{tr}$ (resp. $\F_p(\!(t)\!)^{tr}$) of $\Q_p$ (resp. $\F_p(\!(t)\!)$), which is of arithmetic significance: 
\begin{Coro} \label{maincor}
Assume Conjecture \ref{R}. Then the field $\Q_p^{tr}$ (resp. $\F_p(\!(t)\!)^{tr}$) is existentially decidable in the language of rings (resp. $L_t$).
\end{Coro} 
Ramification theory provides us with nice explicit descriptions for the fields of Corollary \ref{maincor} (see Fact \ref{expl}). 
More generally, we deduce Corollary \ref{dec1}, which can be used to generate plenty of other existentially decidable infinitely ramified fields, both in mixed and positive characteristic (assuming Conjecture \ref{R}). For instance, if $\ell\neq p$ is a prime, we get that $\Q_p(p^{1/\ell^{\infty}})$ (resp. $\F_p(\!(t)\!)(t^{1/\ell^{\infty}})$) is existentially decidable in the language of rings (resp. $L_t$). We note that the case of $\Q_p(p^{1/p^{\infty}})$ (resp. $\F_p(\!(t)\!)^{1/p^{\infty}}$), which was mentioned before, does not appear to be amenable to the methods presented here and probably requires some new ideas.

Apart from decidability consequences, certain new existential Ax-Kochen/Ershov phenomena become accessible (see \S \ref{tweakabh}). For example, if $\ell$ is a prime different from $p$, one can deduce (modulo Conjecture \ref{R}) that $\F_p(\!(t)\!)(t^{1/\ell^{\infty}}) \preceq_1 \F_p(\!(t^{\Gamma_{\ell}})\!)$, the latter being the Hahn series field with value group $\Gamma_{\ell}=\frac{1}{\ell^{\infty}}\Z$ and residue field $\F_p$. This should be contrasted with the classical  example by Abhyankar, namely the Artin-Schreier equation $x^p-x-1/t=0$. This has a solution in the Hahn field $\F_p(\!(t^{\Gamma_p})\!)$ with value group $\Gamma_p=\frac{1}{p^{\infty}}\Z$, namely $x=\sum_{n=1}^{\infty} t^{-1/p^n}$, but has no solution in $\F_p(\!(t)\!)^{1/p^\infty}$. Similarly, we have that $\Q_p(p^{1/\ell^{\infty}})$ is existentially closed in any of its maximal immediate extensions.


\section*{Notation}
\begin{itemize}

\item Let $X$ be a scheme over a discrete valuation ring $R$ with residue field $\kappa$ and fraction field $K$. Let $s$ be the closed point of $\Spec R$ and $\eta$ be its generic point. We denote by
$X_s$ the special fiber $X\times_{\Spec R} \Spec (\kappa)$ and by $X_K$ the generic fiber $X\times_{\Spec R} \Spec K$. 
\item Given an $R$-algebra $A$, we denote by $X(A)$ the set of $A$-integral points, i.e., the set of morphisms $\Spec(A)\to X$ over $\Spec R$. If $X=\Spec (B)$, where $B$ is a finitely generated $R$-algebra of the form $B=R[x_1,...,x_n]/(f_1,...,f_m)$, this can be identified with the set of tuples $(a_1,...,a_n)\in A^n$ such that $f_1(a_1,...,a_n)=...=f_m(a_1,...,a_n)=0$. 
\item If $A$ is an integral domain with $\text{Frac}(A)=L$ and is flat as an $R$-module, we also speak of the underlying $L$-rational point of an $A$-integral point $P: \Spec(A)\to X$, which is simply the induced morphism $\Spec(L)\to X_K$. 
\item If $(K,v)$ is a valued field, we denote by $\mathcal{O}_K$ the valuation ring, $\Gamma$ the value group and $k$ the residue field. 

\item Given a language $L$, an $L$-structure $M$ and $A\subseteq M$, we write $L(A)$ for the language $L$ enriched with constant symbols for the elements of $A$. We write $\text{Diag}_M(A)$ for the atomic diagram of $A$ in $M$.

\item Given $L$-structures $M,N$ with a common substructure $A$, we use the notation 
$M\equiv_{\exists,A} N$ in $L$ to mean that the structures $M$ and $N$ are existentially elementary equivalent in the language $L(A)$.

\item We also introduce the following notation:\\
$L_{\text{rings}}$: The language of rings, i.e. $\{0,1,+,\cdot\}$.\\
$L_t$: The language $L_{\text{rings}}$, together with a constant symbol $t$, whose intended interpretation will always be clear from the context. \\
$L_{\text{oag}}$: The language of ordered abelian groups, i.e. $\{0,+,<\}$.\\
$L_{\text{val}}$: The language of valued fields, construed as a three-sorted language with sorts for the valued field, the value group, the residue field and symbols for the valuation and residue maps. \\
$L_{\text{AKE}}$: The Ax-Kochen/Ershov language; this is the language $L_{\text{val}}$ together with a function symbol $s$ for a cross-section of the valuation, i.e., a right inverse of $v:K^{\times}\to \Gamma$.

\end{itemize}
\section{Preliminaries from geometry } \label{prelims}
Since our approach is purely geometric, following the philosophy of \cite{Den} and \cite{Den2}, we shall review some basic concepts and facts from algebraic geometry. Let $R$ be a DVR with uniformizer $\pi$ and $S=\Spec R$. The basic examples to have in mind throughout the paper are $R=\Z_{(p)}$ or $\Z_p$ with $\pi=p$ and $R=\F_p[t]_{(t)}$ or $\F_p[\![t]\!]$ with $\pi=t$. 

\subsection{Regular schemes}

\begin{definition}
Let $(A,\mathfrak{m})$ be a Noetherian local ring with residue field $k=A/\mathfrak{m}$. We say that $A$ is regular if $\dim_k \mathfrak{m}/\mathfrak{m}^2=\dim A$.
\end{definition}
The concept of a "non-singular" scheme is formalized in the following:
\begin{definition}
Let $X$ be a locally Noetherian scheme.\\
$(a)$ We say that $X$ is regular at $x\in X$, or that $x$ is a regular point of $X$, if $\mathcal{O}_{X,x}$ is a regular local ring. \\
$(b)$ We say that $X$ is regular if it is regular at all $x\in X$.
\end{definition}
If $X$ is regular at $x\in X$, then a minimal set of generators for $\mathfrak{m}_{X,x}$ is said to be a \textit{regular system} of parameters of $X$ at $x$. 


\subsection{Normal Crossings}
\subsubsection{Definition}
A divisor $D$ is said to be \textit{strict normal crossings} if Zariski locally $D_{\text{red}}$ is the union of non-singular hypersurfaces crossing transversely. More formally:
\begin{definition} \label{defnc}
$(a)$ Let $X$ be a locally Noetherian scheme and $D$ be an effective Cartier divisor on $X$. We say that $D$ has \textit{strict} normal crossings at a point $x\in X$ if $X$ is regular at $x$ and there exists a regular system of parameters $f_1,...,f_n$ of $X$ at $x$, an integer $0\leq m \leq n$ and integers $e_1,...,e_m\geq 1$ such that $D$ is cut out by $f_1^{e_1}\cdot f_2^{e_2} \cdot ... \cdot f_m^{e_m}$ in $\mathcal{O}_{X,x}$. If $D$ has strict normal crossings at all points $x\in X$, then $D$ is a strict normal crossings divisor.\\
$(b)$ Let $X$ be a regular scheme and $X\to \Spec R$ be a morphism to a discrete valuation ring $R$. We say that $X\to \Spec R$ has strict normal crossings at $x\in X$ if $X_s=X\times_{\Spec R} \Spec( R/\mathfrak{m}_R)$ has strict normal crossings at $x$. If $X\to \Spec R$ has strict normal crossings at all $x\in X$, we simply say that $X\to \Spec R$ has strict normal crossings.
\end{definition}
The integers $e_i$ in Definition \ref{defnc}$(a)$ are the \textit{multiplicities} of the irreducible components of $D$ passing through $x$. Some examples are given below:
\begin{example} \label{ncdex}
$(a)$ Let $R=\F_p[\![t]\!]$ and $X$ be the affine $R$-scheme defined by $x^2y-t=0$. Then $X\to \Spec R$ has strict normal crossings at $(x,y,t)$. The irreducible components of $X_s=\Spec(\F_p[x,y]/(x^2y) )$ have multiplicities $2$ and $1$.  \\
$(b)$ Let $R=\Z_{(p)}$ and $X$ be the affine $R$-scheme defined by $x^py-p=0$. Then $X\to \Spec R$ has strict normal crossings at $(x,y,p)$. The irreducible components of $X_s=\Spec(\F_p[x,y]/(x^py) )$ have multiplicities $p$ and $1$.
\end{example}

\subsubsection{Normal crossings vs strict normal crossings} \label{ncvssnc}
Since the term "normal crossings" is used with subtly different meanings throughout the literature, we find it useful to clarify a few things. In addition to Definition \ref{defnc}, one has the following more general notion:
\begin{definition} \label{ncdef}
If $X$ is a regular scheme and $D$ is an effective Cartier divisor on $X$, we say that $D$ has normal crossings if there exists an \'etale morphism $\pi:Z\to X$ such that the pullback $\pi^*D$ has \textit{strict} normal crossings. 
\end{definition}
However, sometimes authors use the term "normal crossings" for Definition \ref{defnc} rather than the more general Definition \ref{ncdef} (see Remark 9.1.7). Thanks to the following well-known fact, this distinction is not going to be important for our purposes:
\begin{rem} [see Proposition 2.2.2 \cite{Nic}, \cite{Con}] \label{remnc}
If $D$ has normal crossings, one can always find a blow-up $f:\tilde{X}\to X$ so that $f^*D$ is a strict normal crossings divisor. 
\end{rem}
Remark \ref{remnc} is proved in Proposition 2.2.2 \cite{Nic} when $\dim X=2$ but the general case is similar. It is explained in detail in \cite{Con}. Note that in Conrad's terminology a strict normal crossings divisor is taken to be reduced (see also the last paragraph of \S \ref{ncvssnc}). Nevertheless, the proof of \cite{Con} applies equally well in the non-reduced context. We illustrate Remark \ref{remnc} with an example: 
\begin{example} \label{examplesncd}
Let $p\neq 2$, $R=\F_p[\![t]\!]$ and $X$ be the affine $R$-scheme defined by
$$x^2-\alpha y^2=t$$
where $\alpha \in \F_p^{\times}-(\F_p^{\times})^2$. Let $\beta\in \overline{ \F}_p$ be such that $\beta^2=\alpha$. After a base change $S'\to S$, corresponding to $R\to \F_p(\beta)[\![t]\!]$,  the pullback of the divisor $X_s$ has defining equation 
$$(x-\beta y)(x+\beta y)=0$$
and is thus a normal crossings divisor of $X_{S'}=X\times_S S'$. Blowing up the ideal $(x,y,t)$ makes $X_s$ into a strict normal crossings divisor, in accordance with Fact \ref{remnc}.
\end{example}
We finally warn the reader that in several places in the literature (e.g. \cite{Con} and 2.2.2 \cite{Tem}), the term strict normal crossings divisor means that the divisor is \textit{reduced}. However, our Definition \ref{defnc} allows non-reduced divisors as well. In fact, the reader should have the non-reduced case in mind as the typical case throughout the paper. 
\subsection{Resolution of singularities} \label{rosection}
\subsubsection{Motivation}
For the rest of the paper, unless otherwise stated, we will assume a certain form of resolution of singularities. Recall that a \textit{resolution of singularities} of a scheme $X$ is a proper and birational morphism $f:X'\to X$, with $X'$ regular (Definition 8.3.39 \cite{Liu}). In practice, one often requires more refined forms of resolution. Before we state precisely the form that we will need, let us first discuss the notion of an excellent DVR. 
\subsubsection{Excellent DVR}
The concept of a quasi-excellent ring first appeared in \S 7.9 \cite{EGA}. Grothendieck showed in 7.9.5 \cite{EGA} that if $R$ is a ring such that every integral scheme of finite type over $R$ admits a resolution of singularities, then $R$ is quasi-excellent. 

We shall not dwell on the precise definition of quasi-excellence, which is rather technical. For the rest of the paper, our base ring $R$ will always be a DVR, in which case $R$ is \textit{excellent} precisely when the field extension $\widehat{K}/K$ is \textit{separable}, where $K=\text{Frac}(R)$ (see Corollary 8.2.40$(b)$ \cite{Liu}). The reader may take this as the definition of an excellent DVR, although this is a consequence of the actual definition. We record the following construction, due to F. K. Schmidt, which will appear again in Remark \ref{counternondef}. It is a classical example of a non-excellent DVR:
\begin{example} [see Examples 2.3.5 \cite{Tem}, 3.1 \cite{defect}] \label{defex}
Let $y=\sum_{i=0}^{\infty} \alpha_i x^i \in \F_p[\![x^p]\!]$ be transcendental over $\F_p(x)$ and let $F=\F_p(x,y)$. Consider the embedding $F\hookrightarrow \F_p(\!(x)\!)$ and the restriction of the $x$-adic valuation on $F$. We have that $\widehat{F}=\F_p(\!(x)\!)$. It is clear that $\widehat{F}/F$ is not separable as $y^{1/p}\in \widehat{F}-F$.
\end{example} 
\subsubsection{Resolution of singularities}
Grothendieck conjectured that the converse of 7.9.5 \cite{EGA} is also true, i.e., every integral quasi-excellent scheme admits a resolution of singularities (see 7.9.6 \cite{EGA}). In practice, one often asks for stronger variants of resolution, some of which are discussed in \S 2.5 \cite{Tem}. These variants are mostly discussed in the characteristic $0$ context, where we have proven results.  However, in 2.3.2 \cite{Tem} Temkin explains that the stronger variants discussed in \S 2.5 are expected to be true for \textit{general} quasi-excellent schemes. In this paper, we will assume the validity of the following statement:
\begin{Conj} [Log-Resolution] \label{R}
Let $X$ be a reduced, flat scheme of finite type over an excellent discrete valuation ring $R$. Then there exists a blow-up morphism $f:\tilde{X}\to X$ in a nowhere dense center $Z\subsetneq X$ such that 
\begin{enumerate}
\item $\tilde{X}$ is a regular scheme.
\item $\tilde{X}_s=\tilde{X}\times_{\Spec R} \Spec(R/\mathfrak{m}_R)$ is a strict normal crossings divisor.
\end{enumerate}
\end{Conj}
Condition (1) is the one predicted by Grothendieck's conjecture. A desingularization of $X$ which also satisfies condition (2) is often called a desingularization of the pair $(X,X_s)$ (see 2.5.3 \cite{Tem})\footnote{In the terminology of 2.5.3 \cite{Tem} a desingularization of $(X,X_s)$ is only required to make $X_s$ into a strict normal crossings divisor but this is not really important in view of Fact \ref{remnc}. We warn the reader that in \cite{Tem} the term "monomial divisor" is used for what we called "normal crossings divisor" in Definition \ref{ncdef}. In Temkin's terminology a normal crossings divisor is required to be reduced (see 2.2.2 \cite{Tem}). Note also that 2.5.3 \cite{Tem} asks for some control over the centers of the blow-ups but this is not going to be important for us.} or a log-resolution of $(X,X_s)$. We emphasize that $X_s$ will  typically be non-reduced. In residue characteristic $0$, one can also make $X_s$ reduced, at the cost of replacing $R$ with a finite extension (see 3.1.4 \cite{Tem}).

\subsubsection{Evidence for Conjecture \ref{R}}
We refer to \S 8.3.4 \cite{Liu}, \S 3 \cite{Tem} for more information on established desingularization results. See also \S 3.6-\S 3.8 \cite{Spivak} for a more up-to-date survey.
\begin{fact} [Residue characteristic $0$] \label{Hiro}
Conjecture \ref{R} is known when the residue field of $R$ is of characteristic $0$ by work of Hironaka (see Main Theorem I, pg. 132 and Corollary 3, pg.146 \cite{Hir}). These results are phrased for varieties over fields of characteristic $0$ but on pg. 151 \cite{Hir} Hironaka explains that similar results hold more generally over quasi-excellent local rings of residue characteristic $0$. This is also explained in the introduction of \cite{Tem2}. The notion of "quasi-excellence" does not actually appear in Hironaka's paper but was later introduced by Grothendieck. 

To obtain Conjecture \ref{R} in residue characteristic $0$, one splits the desingularization of the pair $(X,X_s)$ into an ordinary desingularization $X'\to X$ and an \textit{embedded} desingularization of $X'_s \subsetneq X'$. Theorem 1.1 \cite{Tem2} proves a more general result for \textit{general} quasi-excellent schemes. We note that the above cited results are stated for integral (rather than reduced) schemes but this is not very important (see Remark 2.3.8 \cite{Tem2}).   
\end{fact}

In positive characteristic, Conjecture \ref{R} and other variants of resolution are widely open. We nevertheless have some partial results:
\begin{fact} [Positive characteristic]
$(a)$ First, in a series of papers Abhyankar proved resolution of singularties for the case of varieties of dimension at most $3$, over an algebraically closed field $k$ of characteristic $p>5$ (see \cite{Abh}). Cutkovsky gave a simplified and self-contained version of Abhyankar's proof in \cite{Cut}. In \cite{CosPil} and \cite{CosPil2}, Cossart and Piltant removed the restriction on the characteristic and generalized it for base fields $k$ satisfying a very mild assumption (namely that $[k:k^p]<\infty$). \\
$(b)$ Lipman showed that reduced, excellent, Noetherian schemes of dimension 2 can be desingularized, but his result does not include a divisor condition (see Theorem 8.3.44 \cite{Liu}). More recently, Cossart and Piltant \cite{CosPil3} proved a strong desingularization result for general separated quasi-excellent schemes of dimension at most $3$. Conjecture \ref{R} for $\dim(X)=3$ can be deduced directly from Theorem 1.1 \cite{CosPil3}, at least when $X$ is separated. In fact, we will only need Conjecture \ref{R} in the case where $X$ is affine (hence separated). 
\end{fact}

\begin{fact}
A.J. de Jong proved a \textit{weaker} statement than Conjecture \ref{R} with alterations, in the case where $R$ is a complete DVR (see Theorem 6.5 \cite{DJ}).
\end{fact}
For a concrete and detailed calculation of a desingularization of $X=\Spec(R[x,y]/(xy-a))$ with $a\in R$, which also makes $X_s$ into a strict normal crossings divisor, see Example 8.3.53 \cite{Liu}.

\section{Valued fields}
\subsection{RV-structures} \label{rvsec}
We now provide an overview of the RV-structures associated to a valued field $(K,v)$. 
\subsubsection{History}
The RV-structures (also known as Krasner's hyperfields due to \cite{Krasner}) first appeared in a model-theoretic context in the work of Scanlon \cite{Scanlon} (by the name \textit{leading terms}) and were studied by his student Flenner in \cite{Flenner}. These structures are simplified versions of Kuhlmann's earlier \textit{amc-congruences}\footnote{See the discussion on pg. 6 \cite{Flenner} for a comparison between the two formalisms. 
}
(see \cite{KuhlQE}), which in turn were inspired by Basarab's foundational work in \cite{Bas}. Basarab \cite{Bas} introduced his \textit{mixed structures} to obtain relative quantifier elimination and relative completeness results for \textit{general} henselian valued fields of characteristic $0$ (see Theorems A, B \cite{Bas}). However, his results require an \textit{infinite} family of such residual structures, which in practice makes it difficult to use for decidability purposes.\subsubsection{Definition} \label{dfn}
Let $(K,v)$ be a valued field with residue field $k$ and value group $\Gamma$. Consider the following short exact sequence of abelian groups
$$ 0\to k^{\times} \xrightarrow{\iota} K^{\times}/(1+\mathfrak{m}) \xrightarrow{v} \Gamma\to 0$$
where $\iota(a)$ is the equivalence class in $K^{\times}/(1+\mathfrak{m})$ of any lift $\tilde{a}\in \Oo_K$ of $a$. We introduce the structure $\text{RV}(K^{\times})=K^{\times}/(1+\mathfrak{m})$, and write $\text{rv}:K^{\times}\to K^{\times}/(1+\mathfrak{m})$ for the natural map that sends $a\in K^{\times}$ to its equivalence class in $K^{\times}/(1+\mathfrak{m})$. 
As with the value group, it will be convenient to include an element $\infty$ in $\text{RV}(K^{\times})$, resulting in $\text{RV}(K)=\text{RV}(K^{\times})\cup \{\infty\}$. We naturally extend $\text{rv}:K^{\times}\to\text{RV}(K^{\times})$ to $\text{rv}:K\to \text{RV}(K)$ by requiring that $\text{rv}(0)=\infty$. 

Apart from its natural multiplicative structure, inherited from $K^{\times}$, one can also equip $\text{RV}(K)$ with additional structure. 
\begin{itemize}
\item We have a ternary relation $\oplus$ for (multi-valued) addition among elements of $\text{RV}(K)$, so that $\oplus(a,b,c)$ holds if there are $x,y,z \in K$ such that $\text{rv}(x)=a,\text{rv}(y)=b,\text{rv}(z)=c$ and $x+y=z$. We shall simply write $a+b=c$ to mean that $c$ is \textit{unique} such that $\oplus(a,b,c)$. 
\item We equip $\text{RV}(K)$ with a binary relation $R(a,b)\iff vx\leq vy$, where $\text{rv}(x)=a$ and $\text{rv}(y)=b$. 
\end{itemize}
Since $\text{rv}(x)=\text{rv}(y)\implies vx=vy$, the relation $R$ is well-defined. It will be harmless and also convenient to write $va\leq vb$ in place of $R(a,b)$. The "higher" RV-structures described in \S 2 \cite{Flenner} will not be important for us.
\subsubsection{Language} 
\begin{definition} \label{rvlangdef}
The language $L_{\text{RV}}$ is a two-sorted language having the following sorts and symbols:\\
$(1)$ a VF-sort, which uses the language of rings $L_{\text{rings}}=\{+,-,\cdot,0,1\}$.\\
$(2)$ an RV-sort, which uses the group language $\{1, \cdot\}$, a constant symbol for $\infty$,
a ternary predicate for $\oplus$ and a binary predicate intended for the relation $va\leq vb$ (see \ref{dfn}). \\
$(3)$ a function symbol $\text{rv}:\text{VF}\to \text{RV}$ for the natural map $\text{rv}:K \to \text{RV}(K)$. Recall from \ref{dfn} that by convention $\text{rv}(0)=\infty$.
\end{definition}


\bl \label{RV-lemma}
Let $(F,v)$ be a valued field. Then the natural inclusion map $\text{RV}(F)\hookrightarrow \text{RV}(F^h)$ (resp. $\text{RV}(F)\hookrightarrow \text{RV}(\widehat{F})$) is an isomorphism.
\el 
\begin{proof}
Since $(F^h,v)/(F,v)$ is immediate, we have the commutative diagram below, consisting of two short exact sequences
\[
  \begin{tikzcd}
   0 \arrow[r] & \kappa^{\times} \arrow[d,"\cong"] \arrow[r] & \text{RV}(F^{\times}) \arrow[d,hook] \arrow[r] &  \Gamma \arrow[d,"\cong"] \arrow[r] & 0\\
0 \arrow[r] &  \kappa^{\times} \arrow[r] & \text{RV}((F^{h})^{\times}) \arrow[r] & \Gamma \arrow[r] & 0
     \end{tikzcd}
\]
By the Five Lemma, it follows that the natural inclusion $\text{RV}(F^{\times})\hookrightarrow \text{RV}((F^h)^{\times})$ is an isomorphism of abelian groups. Finally, note that this isomorphism respects the additional structure. The same proof applies verbatim for the completion $(\widehat{F},v)$.
\end{proof}
\subsubsection{Cross-section} \label{cross-section}
A right inverse $s:\Gamma \to K^{\times}$ of $v:K^{\times} \to \Gamma$ is called a cross-section and makes the exact sequence of \ref{dfn} split. Therefore, if such a cross-section exists, it allows us to identify the abelian group $\text{RV}(K^{\times})$ with $k^{\times} \times \Gamma$. Explicitly, we identify $(a,\gamma)\in k^{\times}\times \Gamma$ with $\iota(a)\cdot s(\gamma)\in \text{RV}(K^{\times})$. 

We add to $k^{\times} \times \Gamma$ an additional symbol $\infty$ and equip $(k^{\times} \times \Gamma) \cup \{\infty\}$ with the following structure: 
\begin{itemize}
\item A multiplication $\cdot$, which restricts to the natural group operation on $k^{\times}\times \Gamma$ and satisfies $\infty \cdot (a,\gamma)=(a,\gamma)\cdot \infty=\infty \cdot \infty =\infty$ for all $a\in k^{\times},\gamma \in \Gamma$. 
\item A ternary relation $\oplus$ for (multi-valued) addition, defined by $\oplus ( (a,\gamma),(b,\delta), (c,\epsilon))$ whenever (1) $\gamma<\delta$ and $c=a, \epsilon=\gamma$ or (2) $\delta<\gamma$ and $c=b,\epsilon=\delta$ or (3) $\gamma=\delta, a+b\neq 0$ and $c=a+b,\epsilon=\gamma$ or (4) $\gamma=\delta, a+b=0$ and $\epsilon>\gamma$ .
\item A binary relation $R((a,\gamma),(b,\delta))\iff \gamma\leq \delta$. 
\end{itemize}

In the presence of a cross-section, the structure $(\text{RV}(K),\oplus,\cdot, \leq ,1,\infty)$ described in \ref{dfn} is isomorphic to the structure $((k^{\times}\times \Gamma)\cup \{\infty\}, \oplus,\cdot,\leq,1,\infty)$ via the identification described above. 

\subsubsection{Interpretability} \label{multres}
In Fact \ref{flennerfact}, we record the observation that $k$ and $\Gamma$ are \textit{quantifier-free} interpretable in $\text{RV}(K)$. We first introduce some notation and make some preliminary remarks, which will also be useful later on: 

Consider a valuation ring $\Oo_K$ with fraction field $K$. We consider the set $\text{RV}(\Oo_K)=\{x\in \text{RV}(K):vx\geq 0\}$ and $\text{RV}(\Oo_K^{\times})=\{x\in \text{RV}(K):vx= 0\}$, equipped with the induced structure from $\text{RV}(K)$ and readily seen to be quantifier-free interpretable in the latter. We note that in \cite{Den2} the structure $\text{RV}(\Oo_K)$ is denoted by $\text{MR}(\Oo_K)$ and is called the structure of multiplicative residues of $\Oo_K$. 

The residue map $\text{res}: \text{RV}(\Oo_K)\to k$, mapping the equivalence class of $a\in \Oo_K$ in $\text{RV}(\Oo_K)$ to $\overline{a}$ (and $\infty$ to $0$), induces a multiplicative isomorphism $\text{RV}(\Oo_K^{\times}) \cong k^{\times}$, which is also a left inverse of $\iota:k^{\times}\to \text{RV}(\Oo_K^{\times})$. The isomorphism $\text{res}:\text{RV}(\Oo_K^{\times}) \xrightarrow{\cong} k^{\times}$ extends to $\text{res}:\text{RV}(\Oo_K^{\times})\cup \{\infty\} \xrightarrow{\cong} k$, by sending $\infty \mapsto 0$. It also respects addition in the following sense: For all for all $a,b,c\in \text{RV}(\Oo_K^{\times})\cup \{\infty\}$, we have that $\oplus(a,b,c)$ if and only if $\text{res}(a)+\text{res}(b)=\text{res}(c)$. Moreover, if $\text{res}(c)\neq 0$, then $c\in \text{RV}(K)$ is unique such that $\oplus(a,b,c)$ and we may simply write $a+b=c$ (see \S \ref{dfn}).

\begin{fact} [cf. Proposition 2.8 \cite{Flenner}] \label{flennerfact}
The residue field $(k,+,\cdot,0,1)$ and the value group $(\Gamma,+,<,0)$ are quantifier-free interpretable in $\text{RV}(K)$ with the structure described in Definition \ref{rvlangdef}(2). Moreover, the interpretations are uniform in $K$.
\end{fact}
\begin{proof}
We start with the residue field. We identify $(k,+,\cdot,0,1)$ with $\text{RV}(\Oo_K^{\times}) \cup \{\infty\}= \{ x\in \text{RV}(K): v(x)=0\} \cup \{\infty\}$, which is quantifier-free definable in $\text{RV}(K)$ by the formula $vx=v1\lor x=\infty$ (formally $vx=v1$ stands for $vx\leq v1\land v1\leq vx$). The identification is given by $\text{res}: \text{RV}(\Oo_K^{\times})\cup \{\infty\} \to k$ as described above. 

We next argue that the value group is quantifier-free interpretable in $\text{RV}(K)$. We introduce the following quantifier-free definable relation $\sim$ on $\text{RV}(K)$: $x\sim y \iff vx=vy$. The value group is identified with the quotient $\text{RV}(K)/\sim$, addition in $\Gamma$ is recovered by multiplication, the order $<$ is recovered by $vx<vy$ (viz. $vx\leq vy\land \lnot vy\leq vx$) and $0\in \Gamma$ is identified with $\infty \in \text{RV}(K)$.
\end{proof}
\section{Tamely ramified extensions} \label{tameprel}
For valued fields $(K,v)$ and $(L,w)$, we shall abbreviate the residue fields by $k$ and $l$ respectively, and the value groups by $\Gamma$ and $\Delta$ respectively. Although the notion of a tamely ramified extension is usually defined within the context of algebraic extensions, it will be essential for us to have an analogue for general (possibly transcendental) valued field extensions. Such a definition is given by Endler on pg. 180 \cite{Endler} and we follow his treatment.
\subsection{Algebraic tamely ramified extensions}
We first review the usual notion of an \textit{algebraic} tamely ramified extension: 
\begin{definition} \label{tamelyramdfn}
$(a)$ A finite valued field extension $(L,w)/(K,v)$ is said to be tamely ramified if $l/k$ is separable and $p\nmid [\Delta:\Gamma]$, where $p=\text{char}(k)$.\\
$(b)$ An algebraic valued field extension $(L,w)/(K,v)$ is said to be tamely ramified if every finite subextension is tamely ramified.
\end{definition}

\begin{example}
$(a)$ The extension $(\Q_p(p^{1/n}),v_p)/(\Q_p,v_p)$ is tamely ramified if and only if $p\nmid n$.\\
$(b)$ Any unramified extension of $\Q_p$ or $\F_p(\!(t)\!)$ is automatically tamely ramified.\\
$(c)$ If $\ell$ is a prime different than $p$, then $\F_p(\!(t)\!)(t^{1/\ell^{\infty}})$ is tamely ramified over $\F_p(\!(t)\!)$.
\end{example}
One can give an explicit description of the \textit{maximal} tamely ramified algebraic extension of $\Q_p$ and $\F_p(\!(t)\!)$, denoted by $\Q_p^{tr}$ and $\F_p(\!(t)\!)^{tr}$ respectively:
\begin{fact} [Corollary 1, pg. 32 \cite{CF}] \label{expl}
We have that:\\
$(a)$ $\Q_p^{tr}=\Q_p^{ur}(\{p^{1/n}:(p,n)=1\})=\Q_p(\{\zeta_n,p^{1/n}:(p,n)=1\})$.\\
$(b)$ $\F_p(\!(t)\!)^{tr}=\F_p(\!(t)\!)^{ur}(\{t^{1/n}:(p,n)=1\})=  \F_{p}(\!(t)\!)(\{\zeta_n,t^{1/n}:(p,n)=1\})$.
\end{fact}
We stress that the notion of a tamely ramified extension defined above is subtly different from the notion of a tame extension in \cite{Kuhl2}. The difference is that tamely ramified extensions are not required to be defectless. In many cases of interest, the base field is defectless (e.g., a local field). For instance, this applies in the treatment of tame ramification given in \cite{Lang}, \cite{CF} and \cite{Ser}. This is also the case for us here (see Fact \ref{excelimpliesdef}). Over such a defectless base, the notions "tamely ramified" and "tame" coincide. We favor the use of tamely ramified rather than tame extensions because the former readily extends in the transcendental context and will be sufficient for our purposes.
\subsection{Transcendental tamely ramified extensions} \label{transtame}
We now extend the notion of a tamely ramified field extension to the context of transcendental valued field extensions. The need for this level of generality is explained in Remark \ref{transisessential}. First, recall that an extension of fields $l/k$ (not necessarily algebraic) is said to be \textit{separable} if every finitely generated subextension $l_1/k$ has a \textit{separating transcendence basis}, i.e., a transcendence basis $T\subseteq l_1$ such that the extension $l_1/k(T)$ is separable algebraic. 
\begin{definition} \label{gentamedef}
A valued field extension $(L,w)/(K,v)$ is said to be tamely ramified if $l/k$ is separable, the quotient group $\Delta/\Gamma$ is $p$-torsion-free, where $p=\text{char}(k)$.
\end{definition}
Note that Definition \ref{gentamedef} specializes to Definition \ref{tamelyramdfn} in the case of an algebraic extension.
\begin{example} \label{exampletame}
$(a)$ Every valued field extension is tamely ramified when the residue characteristic is zero.\\
$(b)$ The valued field extension $(\Q_p(p^{1/n}),v_p)/(\Q,v_p)$ is tamely ramified if and only if $p\nmid n$. \\
$(c)$ Let $\F_p(\!(t^{\Gamma})\!)$ be the Hahn series field with residue field $\F_p$ and value group $\Gamma$. Suppose $1\in \Gamma$ is a distinguished positive element and write $t$ for $t^1$. The valued field extension $(\F_p(\!(t^{\Gamma})\!),v_t)/(\F_p(t),v_t)$ is tamely ramified if and only if $1$ is not $p$-divisible in $\Gamma$.
\end{example}
For the rest of the paper, unless otherwise stated, a valued field extension is said to be tamely ramified if it is tamely ramified in the generalized sense of Definition \ref{gentamedef}.

\subsection{Some comments on the defect} \label{defectcomment}

\subsubsection{Definition}
First let us recall some basic definitions. See \S 1 \cite{KuhlmannStab} for a more detailed overview. Given a valued field $(K,v)$ and a finite field extension $L/K$ of degree $n$, let $w_1,...,w_r$ be the extensions of $v$ to $L$. We then have the fundamental inequality 
$$n\geq \sum_{i=1}^r e_i\cdot f_i $$
where $e_i=e(w_i/v)$ and $f_i=f(w_i/v)$ are the ramification and inertia degrees associated to $w_i/v$. We say that the extension $L/K$ is \textit{defectless} if equality holds. We also say that $(K,v)$ is \textit{defectless} if any finite extension as above is defectless.
\subsubsection{Excellent vs Defectless}
Let $(F,v)$ be a valued field with valuation ring $\Oo_F$. We say that $(F,v)$ is \textit{separably defectless} if any finite separable extension $F_1/F$ is defectless. Likewise $(F,v)$ is \textit{inseparably defectless} if any finite purely inseparable extension $F_1/F$ is defectless. 
\begin{fact} [Theorem 3.3.5 \cite{engprest}] \label{discvalsepdef}
Any discrete valued field is separably defectless.
\end{fact}
However, when $\text{char}(F)>0$, the valued field need not be inseparably defectless: 
\begin{example}
Let $(F,v)$ be as in Example \ref{defex}. We see that $F(y^{1/p})/F$ is an inseparable immediate extension. In particular, it is not defectless.
\end{example}
Nevertheless, in the context of Theorem \ref{mainA}, our base ring $\Oo_F$ will in fact be an \textit{excellent} DVR. In this case, we record the following fact: 
\begin{fact} \label{excelimpliesdef}
Let $(F,v)$ be a discrete valued field. The following are equivalent: 
\begin{enumerate}
\item $(F,v)$ is defectless. 
\item $\Oo_F$ is excellent.
\end{enumerate}
\end{fact} 
\begin{proof}
By Corollaire 7.6.6 \cite{EGA}, $\Oo_F$ is excellent if and only if it is a Japanese ring, which is equivalent to $(F,v)$ being defectless by  Th\'eor\`eme 2, pg. 143 \cite{Bourbaki}. 
\end{proof}

\section{RV-Hensel's Lemma}
\subsection{Hensel's Lemma}
\subsubsection{Motivation} 
Let $X\to \Spec R$ be a \textit{smooth} morphism, where $R$ is a henselian local ring $R$ with residue field $\kappa$. The \textit{classical} geometric version of Hensel's Lemma allows us to lift $\kappa$-rational points of $X_s$ to $R$-integral points of $X$ (see e.g. Corollary 6.2.13 \cite{Liu}). Proposition \ref{RV-HENSEL} is an analogue of this, for the case where $X\to \Spec R$ is not necessarily smooth but has \textit{strict normal crossings}. In that case, one may lift RV-points of $X$ to integral points of $X$, at least when one of the multiplicities of the irreducible components of $X_s$ is not $p$-divisible, where $p=\text{char}(\kappa)$. For lack of a suitable reference, we shall spell out the details.
\subsubsection{Hensel's Lemma for smooth morphisms}
Recall the geometric version of Hensel's Lemma for (locally) smooth morphisms: 
\bl [Hensel's Lemma for smooth morphisms] \label{henetale}
Let $\pi:X\to Y$ be a morphism of schemes over a henselian local ring $R$ with residue field $\kappa$. Let $P\in Y(R)$ and suppose that the induced $\kappa$-rational point $y\in Y_s(\kappa)$ lifts to $ x\in X_s(\kappa)$ and suppose that $\pi$ is smooth at $x$. Then $P$ lifts to an $R$-integral point of $X$, which specializes to $x$. 
\el 
\begin{proof}
Let $S=\Spec R$ and $s$ be its unique closed point. Consider the fiber product
\[
  \begin{tikzcd}
   Z \arrow[d] \arrow[r] & X \arrow[d,"\pi"]\\
   S  \arrow[r] & Y
     \end{tikzcd}
\]
where $Z=X\times_Y S$ and $S\to Y$ is the morphism corresponding to the point $P\in Y(R)$. Note that $Z\to S$ is locally smooth by base change. Working now as in Corollary 6.2.13 \cite{Liu}, we get an $R$-integral point of $X$ lifting $x$.
\end{proof}
\begin{rem}[Remark 6.2.14 \cite{Liu}]
Corollary 6.2.13 \cite{Liu} is stated with a completeness assumption but henselianity is enough (cf. Proposition 2.3.5 \cite{Bosch}).
\end{rem}

\subsection{RV-Hensel's Lemma} \label{rvhensec}
Throughout this section, the ring $R$ will be a DVR with uniformizer $\pi$ and residue field $\kappa=R/\pi R$. We set $p=\text{char}(\kappa)$. We also have a \textit{henselian} valuation ring $A$ extending $R$ (typically non-discrete). We let $\mathfrak{m}_A$ be the maximal ideal and $k=A/\mathfrak{m}_A$ be the residue field of $A$.  Suppose $X$ is a scheme over $R$ and $x\in X_s$. Let $\overline{x}\in X_s(k)$ be a $k$-rational point above $x$. This means that we may identify the residue field $\kappa(x)=\Oo_{X,x}/\mathfrak{m}_{X,x}$ with a subfield of $k$. Given $a \in \Oo_{X,x}$, we write $\overline{a}$ for the image of $a$ inside $k$ via this identification. Finally, recall the map $\iota:k^{\times} \to \text{RV}(K^{\times})$ from \S \ref{dfn}. 

Proposition \ref{RV-HENSEL}---which we call RV-Hensel's Lemma---is the key ingredient in the proof of Theorem \ref{mainA}. The proof of RV-Hensel's lemma goes via an \'etale local analysis of strict normal crossings around a point in the special fiber, eventually reducing to Lemma \ref{simplcase}. 
The case $l=0$ in Lemma \ref{simplcase} simply means that there are no $T_i$'s and gives the \'etale local picture around a closed point of the special fiber. The presence of the $T_i$'s is precisely to account for the non-closed case. 
\bl  \label{simplcase}
Let $0\leq n\leq m$ and $l \geq 0$ be integers. Let $\overline{F}(T)\in \kappa[T_1,...,T_l][T]$ be separable and irreducible over $\kappa(T_1,...,T_l)$ and $F(T)\in R[T_1,...,T_l][T]$ be any lift. Consider the affine scheme 
$$Y=\Spec(R[X_1,...,X_m,T_1,...,T_l,Z]/( Z\cdot X_1^{e_1}\cdot...\cdot X_n^{e_n}-\pi, F(Z) ))$$ 
where $e_i\in \Z^{>0}$ for $i=1,...,n$ and $p\nmid e_1$. Let $y\in Y_s$ be the point corresponding to the prime ideal $(\pi,X_1,...,X_m)$. Let $\overline{y}\in Y_s(k)$ be a $k$-rational point above $y$ and suppose that there exist $a_1,...,a_n\in \mathfrak{m}_A$ such that $\iota(\overline{Z})\cdot \text{rv}( a_1^{e_1}\cdot...\cdot a_n^{e_n})= \text{rv}(\pi) $. Then there exists $P\in Y(A)$ specializing to $\overline{y}$.
\el 
\begin{proof}
The point $\overline{y}\in Y_s(k)$ corresponds to a morphism of $R$-algebras 
$$R[X_1,...,X_m,T_1,...,T_l,Z]/( Z\cdot X_1^{e_1}\cdot...\cdot X_n^{e_n}-\pi, F(Z)) \to k$$
with kernel $(\pi,X_1,...,X_m)$. Write $t_i=\overline{T_i}$ for the image of $T_i$ in $k$ and note that the $t_i$'s are algebraically independent. Let $\tilde{t}_i\in A$ be an arbitrary lift of $t_i$ and $\tilde{f}(T)$ be the polynomial $F(T)$ where we replace $T_i$ with $\tilde{t}_i$. Likewise, the reduction $f(T)$ of $\tilde{f}(T)$ modulo $\mathfrak{m}_A$ is the same as $\overline{F}(T)$ with $T_i$ replaced by $t_i$. It is also separable since $\kappa(t_1,...,t_l)\cong \kappa(T_1,...,T_l)$ and $f(\overline{Z})=0$.
Hensel's lemma allows us to choose $u \in A^{\times}$ so that $\tilde{f}(u)=0$ and $\overline{u}=\overline{Z}$. By assumption, there exist $a_1,...,a_n\in A$ such that $u\cdot a_1^{e_1}\cdot...\cdot a_n^{e_n} = \pi \cdot  \varepsilon$, with $\varepsilon\in 1+\mathfrak{m}$. Since $p\nmid e_1$, Hensel's lemma provides us with $\alpha\in A^{\times}$ such that $\alpha^{e_1}=\varepsilon$. Replacing $a_1$ with $a_1\cdot \alpha$ gives us an $A$-integral point $P\in Y(A)$ corresponding to a morphism of $R$-algebras
$$R[X_1,...,X_m,T_1,...,T_l,Z]/( Z\cdot X_1^{e_1}\cdot...\cdot X_n^{e_n}-\pi, F(Z) )\to A$$
mapping $X_i\mapsto a_i$, $T_i\mapsto \tilde{t}_i$ and $Z\mapsto u$. One readily checks that $P$ specializes to $\overline{y}$.
\end{proof}
Recall the definition of strict normal crossings in Definition \ref{defnc}$(b)$. 
\bp [RV-Hensel's Lemma]  \label{RV-HENSEL}
Let $X\to \Spec R$ be a finite type morphism having strict normal crossings at $x\in X_s$. Write $\pi =h \cdot x_1^{e_1}\cdot ...\cdot x_n^{e_n}$ in $\Oo_{X,x}$, where $e_i \in \Z^{>0}$, $h\in \Oo_{X,x}^{\times}$ and $\{x_1,...,x_n\}$ is part of a regular system of parameters of $X$ at $x$. We further assume that $\kappa(x)/\kappa$ is separable and $p\nmid e_1$.  Suppose there exists $\overline{x}\in X_s(k)$ above $x$ and $a_1,...,a_n \in \mathfrak{m}_A$ such that $\iota(\overline{ h})\cdot \text{rv}(a_1^{e_1}\cdot ...\cdot a_n^{e_n})=\text{rv}( \pi)$ in $\text{RV}(A)$. Then there exists $P\in X(A)$ specializing to $\overline{x}$. 
\ep  
\begin{proof}
Shrinking $X$ around $x$, if necessary, we may assume that $X=\Spec (B)$. We may also assume that the $x_i$'s are regular functions on all of $X$. Let $\mathfrak{p}\subsetneq B$ be the prime ideal corresponding to $x$. Since $\kappa(x)/\kappa$ is finitely generated and separable, we may choose a separating transcendence basis, say $\{t_1,...,t_l\}$ (or $\emptyset$ in case $x\in X_s$ is closed). Let $f(T)\in \kappa[t_1,...,t_l][T]$ be irreducible and separable over $\kappa(t_1,...,t_l)$ with $f(\overline{ h})=0$. Let $\{\tilde{t}_1,...,\tilde{t}_l\}$ be a set of lifts in $B_{\mathfrak{p}}$. By shrinking $X$ further, one may even assume that $\tilde{t}_1,...,\tilde{t}_l\in B$. Let $\tilde{f}(T)\in R[\tilde{t}_1,...,\tilde{t}_l][T]$ be a lift of $f(T)$. \\
\textbf{Claim:} There is a common \'etale neighborhood $(U,u)$ of $(X,x)$ and $(Y,y)$, where $(Y,y)$ is as in Lemma \ref{simplcase}.  Moreover, there exists $\overline{u} \in U_s(k)$ above $u$.
\begin{proof}
We pass to an \'etale neighborhood $(U,u)$ of $(X,x)$, where $U=\Spec (C)$ with $C=B[z]/(\tilde{f}(z))$ and $u$ is the point of $U_s$ corresponding to $\mathfrak{q}=(\mathfrak{p},z-h)$. To verify the \'etaleness of $(U,u)\to (X,x)$ at $u$, observe that $\tilde{f}'(z) \notin \mathfrak{q}$. Indeed, otherwise we would have $f'(\overline{h})=0$, which is contrary to the assumption that $f(T)$ is separable. Finally, note that $(U,u)\to (X,x)$ is residually trivial, i.e., $\kappa(u)=\kappa(x)$. By Tag 02GU(8) \cite{sp}, we see that $(U,u)\to (X,x)$ is \'etale at $u$. In particular, we get that $U$ is regular at $u$ by Tag 025N \cite{sp}. 

We introduce $Y=\Spec(R[X_1,...,X_m,T_1,...,T_l,Z]/( Z\cdot X_1^{e_1}\cdot...\cdot X_n^{e_n}-\pi, F(Z) ))$, where $F(T)\in R[T_1,...,T_l][T]$ is the same as $\tilde{f}(T)$ but with $\tilde{t}_i$ replaced with $T_i$. Consider the morphism $U \to Y$, which corresponds contravariantly to the ring homomorphism sending $X_i\mapsto x_i$, $T_i\mapsto \tilde{t}_i$ and $Z\mapsto z$. Note that $Y$ is regular at $y\in Y_s$ associated to the prime ideal $(\pi,X_1,...,X_m)$ and that $U\to Y$ maps $u$ to $y$. 

We claim that $(U,u)\to (Y,y)$ is \'etale at $u$. First, note that $\{X_1,...,X_m\}$ is a regular system of parameters of $Y$ at $y$. This maps via $\Oo_{Y,y}\to \Oo_{U,u}$ to $\{x_1,x_2,...,x_m\}$, which is a regular system of parameters for $U$ at $u$. We also need to argue that $\kappa(u)/\kappa(y)$ is finite separable. Indeed, we have a tower of finite extensions $\kappa(u)/\kappa(y)/\kappa(t_1,...,t_l)$ with $\kappa(u)/\kappa(t_1,...,t_l)$ being finite separable. It follows that $\kappa(u)/\kappa(y)$ is also finite separable.  \'Etaleness of $(U,u)\to (Y,y)$ may now be verified using Lemma 2.1.4 \cite{Nic}. Finally, the point $\overline{x}\in X_s(k)$ together with the fact that $(U,u)\to (X,x)$ is residually trivial at $u$, furnish us with a rational point $\overline{u}\in U_s(k)$ above $u$.
\qedhere $_{\textit{Claim}}$ \end{proof}
We summarize the above in the following diagram
\[
\begin{tikzcd}
 & (U,u) \arrow{dr}{\text{\'et}} \arrow[swap]{dl}{\text{\'et}} \\
 (Y,y)  && (X,x) 
\end{tikzcd}
\]
Note that there is $\overline{y}\in Y_s(k)$ above $y$, corresponding to the morphism 
$$\kappa[X_1,...,X_m,T_1,...,T_l,Z]/( Z\cdot X_1^{e_1}\cdot...\cdot X_n^{e_n}, F(Z) ) \to k $$ 
with $X_i \mapsto 0, T_i\mapsto t_i$ and $Z\mapsto \overline{h}$ (i.e., $\overline{Z}=\overline{h}$). Using our assumption, we may find $a_1,...,a_n \in \mathfrak{m}_A$ such that $\iota(\overline{Z}) \cdot \text{rv}( a_1^{e_1}\cdot ...\cdot a_n^{e_n})=\iota(\overline{h}) \cdot \text{rv}( a_1^{e_1}\cdot ...\cdot a_n^{e_n})= \text{rv}(\pi) $ in $\text{RV}(A)$. Lemma \ref{simplcase} implies the existence of an $A$-integral point of $Y$ specializing to $\overline{y}$. Consider the morphism $U_A\to Y_A$, induced from $U\to Y$ by base change. The morphism $U_A\to Y_A$ is \'etale at $u_A$ by base change. By Lemma \ref{henetale}, we get an $A$-integral point of $U$ specializing to $\overline{u}$. Finally, the latter induces $P\in X(A)$ specializing to $\overline{x}$.
\end{proof}
The assumption that one of the multiplicities not be $p$-divisible is necessary: 
\begin{example}
Let $R=\F_p[\![t]\!]$ and $A=\F_p[\![t^{1/p}]\!]$. Consider 
$$X=\Spec(R[x]/((1+x)\cdot x^p-t)$$ 
Let $P$ be the point at the origin of the special fiber, corresponding to the maximal ideal $(x,t)$. Although, $P\in X_s(\F_p)$ and $a=t^{1/p}$ has the property that $\text{rv}((1+a)\cdot a^p)=\text{rv}(t)$, one sees that $X$ has no $A$-integral point lifting $P$.  Indeed, this would yield a solution of $(1+x)\cdot x^p=t^p$ in $\F_p[\![t]\!]$ with $v_tx>0$. This would then imply that $1+x\in \F_p[\![t^p]\!]$ and therefore $x\in \F_p[\![t^p]\!]$. Setting $x=y^p$ and taking $p$-th roots, we get $(1+y)\cdot y^p=t$, which is impossible since the left hand side has $p$-divisible valuation.
\end{example}
\section{Theorem \ref{mainA}}
\subsection{Restating Theorem \ref{mainA}} 
\begin{Theore} \label{mainA}
Assume Conjecture \ref{R}. Suppose $(K,v)$ and $(L,w)$ are henselian and tamely ramified over a valued field $(F,v_0)$ with $\Oo_F$ an excellent DVR. If $\text{RV}(K)\equiv_{\exists,RV(F)} \text{RV}(L)$, then $K\equiv_{\exists,F} L$ in $L_{\text{rings}}$.
\end{Theore} 
For the valued fields $(K,v)$ and $(L,w)$, we shall abbreviate the residue fields by $k$ and $l$ respectively, and the value groups by $\Gamma$ and $\Delta$ respectively. We write $\pi$ for a uniformizer of $\Oo_F$, $\kappa=\Oo_F/\pi \Oo_F$ for the residue field and $p=\text{char}(\kappa)$. 

\subsection{Geometric reformulation} \label{geometricreform}
We now reformulate our task in geometric terms. Using the disjunctive normal form and replacing conjunctions $\bigwedge_{i=1}^n f_i(X_1,...,X_m)\neq 0$ with a single inequation $\prod_{i=1}^n f_i(X_1,...,X_m)\neq 0$, every existential $L_{\text{rings}}(\Oo_F)$-sentence is equivalent to a disjunction of sentences of the form
$$\phi=\exists X_1,...,X_m (f_1(X_1,...,X_m)=...=f_n(X_1,...,X_m)=0 \mbox{ and } g(X_1,...,X_m)\neq 0) $$
for some $f_i(X_1,...,X_m),g(X_1,...,X_m) \in \Oo_F[X_1,...,X_m]$ and $n,m\in \N$.
It is enough to focus on one such disjunct. We now define the affine $\Oo_F$-algebra
$$B=\Oo_F[X_1,...,X_m]/(f_1(X_1,...,X_m),...,f_n(X_1,...,X_m))$$ 
and let $X=\Spec (B)$ be the associated affine scheme over $\Spec(\Oo_F)$. Let $W\subseteq X_F$ be the (basic) Zariski open subset of the generic fiber, defined by the extra condition $g(X_1,...,X_m)\neq 0$. A witness of $\phi$ in $\Oo_K$ (resp. $\Oo_L$) corresponds to an integral point $P\in X(\Oo_K)$ whose underlying $K$-rational point satisfies $P_K \in W(K)$. 

From the above discussion and by symmetry, it suffices to prove the following:
\begin{Theorema} \label{theorema'}
Assume Conjecture \ref{R}. Let $(K,v)$ and $(L,w)$ be valued field extensions of $(F,v_0)$ with $\Oo_F$ an excellent DVR.  Suppose that $(K,v)/(F,v_0)$ is tamely ramified, $(L,w)$ is henselian and that  $\text{RV}(L)\models \text{Th}_{\exists,\text{RV}(F)} \text{RV}(K)$. Let $X$ be a scheme of finite type over $\Oo_F$ and $W\subseteq X_F$ be a Zariski open subset. Suppose there exists $P\in X(\Oo_K)$ with $P_K\in W(K)$. Then there also exists $Q\in X(\Oo_L)$ with $Q_L\in W(L)$.
\end{Theorema}
Let also $x\in X_s$ be the scheme-theoretic point (not necessarily closed) where $P$ meets the special fiber. The proof of Theorem \ref{theorema'} is divided into two steps: 
\begin{enumerate}
\item We treat the case where $X\to \Spec \Oo_F$ has strict normal crossings at $x$ and $W=X_F$.

\item We treat the general case using Conjecture \ref{R} and an inductive argument on $\dim(X)$. 
\end{enumerate}
\subsubsection{Strict normal crossings} \label{snccase}
First assume that $W=X_F$ and $X\to \Spec \Oo_F$ has strict normal crossings. Since $X\to \Spec \Oo_F$ has strict normal crossings at $x$, there exists a regular system of parameters $\{x_1,...,x_m\}$ in $\Oo_{X,x}$ such that $h\cdot x_1^{e_1}\cdot ...\cdot x_n^{e_n}=\pi $, where $\pi$ is a uniformizer of $\Oo_F$ and $h\in \Oo_{X,x}^{\times}$. 

The point $P$ corresponds to a local $\Oo_F$-algebra homomorphism $\Oo_{X,x}\to \Oo_K$ mapping $f\mapsto f(P)$. Taking valuations in the equation 
$$h(P)\cdot x_1^{e_1}(P)\cdot ...\cdot x_n^{e_n}(P)=\pi$$ 
and noting that $h(P)$ is a unit in $\Oo_K$, yields 
$$\sum_{i=1}^n e_i\cdot vx_i(P) = v\pi$$
Since $\Gamma/\Z v\pi$ has no $p$-torsion elements, we get that $p\nmid e_i$, for some $i\in \{1,...,n\}$. Suppose $p\nmid e_1$, without loss of generality. Moreover, we have that $\kappa(x)/\kappa$ is separable because $k/\kappa$ is separable and $\kappa\subseteq \kappa(x)\subseteq k$. Now we may assume that $L$ is $|K|^+$-saturated as an $L_{\text{RV}}$-structure. Since $\text{RV}(L)\models \text{Th}_{\exists, \text{RV}(F)} \text{RV}(K)$, this furnish us with an embedding $\rho:\text{RV}(\Oo_K)\hookrightarrow \text{RV}(\Oo_L)$ over $\text{RV}(\Oo_F)$ and a compatible embedding of residue fields $\overline{\rho}:k\to l$. In particular, the point $\overline{x}$ induces a rational point $\overline{x}' \in X_s(l)$ above $x$ and there exist $a_1,...,a_n\in \mathfrak{m}_L$ such that $\iota(\overline{\rho}(\overline{h}))\cdot \text{rv}(a_1^{e_1}\cdot ...\cdot a_n^{e_n})=\text{rv}(\pi)$. Proposition \ref{RV-HENSEL} applies to give us $Q\in X(\Oo_L)$ meeting $X_s$ at $x$.
\subsubsection{Reduction to strict normal crossings} \label{generalcasethma}
Our analysis of the general case follows the line of reasoning of the proof of Theorem 4.3 \cite{Den}. Let $X$ be a scheme of finite type over $\mathcal{O}_F$ and $W\subseteq X_F$ be Zariski open. We shall argue by induction on $\dim(X)$ 
that if there exists $P\in X(\Oo_K)$ with $P_K\in W(K)$, then there also exists $Q\in X(\Oo_L)$ with $Q_L\in W(L)$. The base case $\dim(X)=0$ holds vacuously because there cannot exist $P\in X(\Oo_K)$ when $\dim(X)=0$. 
 We only need to explain how the inductive step works.

We shall argue that it is enough to assume that $X$ is integral and affine. We may first assume that $X$ is reduced by passing to its reduced underlying scheme. To see this, note that $\Oo_K$ is reduced and therefore the integral point $P: \Spec(\Oo_K)\to X$ factors (uniquely) as $\Spec(\Oo_K)\xrightarrow{P'} X_{\text{red}}\to X$. 
We have an equality $|X|=|X_{\text{red}}|$ of underlying topological spaces and $P'(\eta_K)=P(\eta_K)$, where $\eta_K$ is the generic point of $\Spec(\Oo_K)$. In particular, we have $P'(\eta_K)\in W$. It would then suffice to prove the inductive step for $X_{\text{red}}$ in place of $X$, so that we may assume $X=X_{\text{red}}$ to begin with. We may then assume that $X$ is irreducible, replacing $X$ with one of its irreducible components which contains the scheme-theoretic image of $P:\Spec (\Oo_K)\to X$. Finally, replacing $X$ with an affine neighborhood allows us to assume that $X=\Spec(B)$, where $B$ is an integral domain, which is a finitely generated $\Oo_F$-algebra with $\Oo_F\to B$ injective (here $\Oo_F\to B$ corresponds to the structure morphism $X\to \Spec(\Oo_F)$). 

It follows that $B$ is a torsion-free $\Oo_F$-module and therefore flat by Tag 0539 \cite{sp}. By Conjecture \ref{R}, we will have a blow-up morphism $f:\tilde{X}\to X$ with $\tilde{X}\to \Spec \Oo_F$ having strict normal crossings and $f$ an isomorphism outside a nowhere-dense closed subscheme $Z\subsetneq X$, the center of $f$. If $P_K\in Z_K(K)$, then $P$ is an integral point of $Z$ and since $\dim(Z)<\dim(X)$, the conclusion follows from our induction hypothesis. Otherwise, we will have that $P_K\in X_F(K)-Z_F(K)$ and $P_K$ lifts to $\tilde{P}_K:\Spec(K)\to \tilde{X}$, using that $f_F:\tilde{X}_F\to X_F$ is an isomorphism outside $Z_F\subsetneq X_F$. By the valuative criterion of properness, the integral point $P$ lifts to an $\Oo_K$-integral point of $\tilde{X}$  as in the diagram below
\[
  \begin{tikzcd}
    \Spec K \arrow[d] \arrow[r,"\tilde{P}_K"] & \tilde{X}  \arrow[d,"f"]\\
   \Spec \Oo_K  \arrow[ur, dotted, "\exists"] \arrow[r, "P"] & X 
     \end{tikzcd}
\]

By the analysis of the strict normal crossings case, one also gets an $\mathcal{O}_L$-integral point of $\tilde{X}$. Now $W-Z_F$ is a Zariski dense open subset of $X_F$. To see this, note that $W, X-Z_F\subseteq X_F$ are non-empty and $X_F$ is irreducible. By Theorem 2.4 \cite{Den}, the $\Oo_L$-integral point of $\tilde{X}$ can be chosen so that its underlying $L$-rational point is in $f_F^{-1}(W)-f_F^{-1}(Z_F)$, which is a Zariski dense open subset of $\tilde{X}_F-f_F^{-1}(Z_F)$, as $f_F$ is an isomorphism outside $Z_F$ and $W-Z_F$ is a Zariski dense open subset of $X_F$. This point induces an $\mathcal{O}_L$-integral point of $Q\in X(\Oo_L)$ with $Q_L\in W(L)$ via composition with $f$, which is what we wanted to show. 
\subsection{Some remarks on Theorem \ref{mainA}}

\begin{rem}
By Fact \ref{Hiro}, Conjecture \ref{R} holds when the residue characteristic is $0$ and the above proof becomes unconditional. In that case, it becomes automatically true that $\Oo_F$ excellent (see Corollary 8.2.40$(c)$ \cite{Liu}) and that the valued field extensions are tamely ramified (see Example \ref{exampletame}(1)). Therefore, one recovers the \textit{existential} version of Ax-Kochen/Ershov in residue characteristic $0$. In \cite{Den2}, Denef manages to recover the full first-order Ax-Kochen/Ershov theorem in residue characteristic $0$ using weak toroidalization of morphisms.
\end{rem}
In the case of \textit{finite} tame ramification in \text{mixed} characteristic and when the residue fields are perfect, the full first-order version of Theorem \ref{mainA} has been proved \textit{unconditionally} by J. Lee:
\begin{rem}[see Corollary 5.9 \cite{Junguk}] \label{jungukrem}
When $(K,v),(L,w)$ are henselian valued fields, finitely and tamely ramified over $(\Q,v_p)$ with perfect residue fields, then J. Lee proves unconditionally that $\text{RV}(K)\equiv \text{RV}(L)$ implies $K\equiv L$ in $L_{\text{rings}}$. 
\end{rem}


Our proof did not use the assumption that $F$ admits no defect extensions inside $K$ (resp. $L$), which is part of the definition of a tamely ramified extension (see \ref{tamelyramdfn}). The reason is that $(F,v_0)$ is automatically defectless (see \S \ref{defectcomment}). If $\Oo_F$ is not assumed to be excellent and $K,L$ are only tamely ramified in the weak sense of Endler \cite{Endler}, then the conclusion of Theorem  \ref{mainA} does not necessarily hold:
\begin{example} \label{counternondef}
Let $R$ be the DVR that was introduced in Example \ref{defex}, for which there exists $\alpha\in R$ such that  $\alpha^{1/p} \in \widehat{R}-R^h$. Set $K=\text{Frac}(R^h)$, $L=\text{Frac}(\widehat{R})$ and $F=\text{Frac}(R)$. We have that $K \not \equiv_{\exists, F} L$, although $\text{RV}(K)\cong_{\text{RV}(F)}\text{RV}(L)$ (see Lemma \ref{RV-lemma}).
\end{example}
%
\section{Applications} \label{app}

\subsection{Decidability} \label{akegen}

\subsubsection{Motivation}
In Remark 7.6 \cite{AnscombeFehm}, the authors write:\\"\textit{At present, we do not know of an example of a mixed characteristic henselian valued field $(K, v)$ for which $k$ and $(\Gamma, vp)$ are $\exists$-decidable but $(K, v)$ is $\exists$-undecidable.}" The existence of such an example is demonstrated in Remark 3.6.9 \cite{KK}. However, if we restrict ourselves to the tamely ramified setting and require that $(K,v)$ admits a cross-section extending a cross-section of $(\Q,v_p)$, we indeed get such an Ax-Kochen/Ershov style statement in Corollary \ref{dec1} (modulo Conjecture \ref{R}). In fact, Corollary \ref{dec1} is stated in a uniform fashion in all characteristics.

\subsubsection{Existential Ax-Kochen/Ershov}
For the sake of simplicity and concreteness, we take our base field $(F,v_0)$ to be any of the valued fields $(\Q(t),v_t),(\Q,v_p)$ and $(\F_p(t),v_t)$. The associated valuation rings are indeed excellent DVRs (see Corollary 8.2.40 \cite{Liu}).  
The equal characteristic $0$ and \textit{unramified} mixed characteristic versions of Corollaries \ref{cor1} and \ref{dec1} are well-known. However, in this level of generality, the mixed characteristic and positive characteristic versions are new. 
\bc \label{cor1}
Assume Conjecture \ref{R}. Suppose $(K,v)$ and $(L,w)$ are henselian valued fields tamely ramified over 
$(\Q,v_p)$ (resp. $(\F_p(t),v_t)$ or $(\Q(t),v_t)$).
Suppose $(K,v)$ and $(L,w)$ admit cross-sections that restrict to the same cross-section of $(\Q,v_p)$ 
(resp. $(\F_p(t),v_t)$ or $(\Q(t),v_t)$). 
If $k\equiv_{\exists} l$ in $L_{\text{rings}}$ and $(\Gamma,vp)\equiv_{\exists} (\Delta,wp)$ in $L_{\text{oag}}$ (resp. $(\Gamma,vt)\equiv_{\exists} (\Delta,wt)$), then $K\equiv_{\exists} L$ in $L_{\text{rings}}$ (resp. $L_t$).
\ec 
\begin{proof} 
We focus on the mixed characteristic case---the proof applies verbatim to the other two cases. Let $(L^*,w^*,s^*_L)$ be a $|K|^+$-saturated elementary extension of $(L,w,s_L)$ in the language $L_{\text{AKE}}$ (see notation). In particular, the residue field $l^*$ and the value group $\Delta^*$ are themselves $|k|^+$-saturated and $|\Gamma|^+$-saturated respectively. Since $l \models \text{Th}_{\exists}k$ and $(\Delta,wp) \models \text{Th}_{\exists}(\Gamma,vp)$, we get a field embedding $\rho:k\hookrightarrow l^*$ and an embedding of pointed ordered abelian groups $\sigma:(\Gamma,vp) \hookrightarrow (\Delta^*,w^*p)$. 

The structure $(\text{RV}(K),\oplus,\cdot, \leq ,1,\infty)$ (resp. $(\text{RV}(L^*),\oplus,\cdot, \leq ,1,\infty)$) described in \ref{dfn} is isomorphic to the structure $((k^{\times}\times \Gamma)\cup \{\infty\}, \oplus,\cdot,\leq,1,\infty)$ (resp. $((l^{* \times}\times \Delta^*)\cup \{\infty\}, \oplus,\cdot,\leq,1,\infty)$) via the identification described in \ref{cross-section}. Moreover, these identifications are compatible with the identification of $(\text{RV}(\Q),\oplus,\cdot, \leq ,1,\infty)$ with $((\F_p^{\times}\times \Z)\cup \{\infty\}, \oplus,\cdot,\leq,1,\infty)$, as the cross-sections of $K$ and $L^*$ extend the one of $\Q$. The maps $\rho$ and $\sigma$ combine to give us an embedding of RV-structures $\text{RV}(K)\hookrightarrow \text{RV}(L^*)$ over $\text{RV}(\Q)$. Reversing the roles of $K$ and $L$, we deduce that $\text{RV}(K)\equiv_{\exists, \text{RV}(\Q)} \text{RV}(L)$ and the conclusion follows from Theorem \ref{mainA}.
\end{proof}
The cross-section condition of Corollary \ref{cor1} cannot be omitted, as the following example shows:
\begin{example} \label{crosseasy}
Take $p\neq  2$ such that $2\notin (\F_p^{\times})^2$ (e.g. $p=3$). Let $(K,v)=(\Q_p(p^{1/2}),v_p)$ and $(L,w)=(\Q_p((2p)^{1/2}),v_p)$. It is clear that $k=l=\F_p$ and $(\Gamma,vp)=(\Delta,wp)\cong (\Z,2)$. On the other hand, we have that $K\not \equiv_{\exists} L$ in $L_{\text{rings}}$. Indeed, if $p^{1/2}\in L$ this would imply that $2\in (\F_p^{\times})^2$.
\end{example}
This should be contrasted with the \textit{unramified} case, where the cross-section condition may be dropped. The equal characteristic $0$ and mixed characteristic parts of Remark \ref{crossrem} are well-known. The (conditional) part $(b)$ of Remark \ref{crossrem} can also be deduced from the method of Denef-Schoutens \cite{Den}, although it is not explicitly mentioned in their paper:
\begin{rem} \label{crossrem}
$(a)$ When $(K,v)$ is \textit{unramified}, meaning that $vp$ (resp. $vt$) is the smallest positive element of $\Gamma$, there is always an elementary extension admitting a cross-section which extends the standard cross-section $n\mapsto p^n$ (resp. $n\mapsto t^n$) of $(\Q,v_p)$ (resp. $(\F_p(t),v_t)$ or $(\Q(t),v_t)$). This follows from Proposition 5.4 \cite{vdd}. In particular, one can drop the cross-section condition in Corollary \ref{cor1} in the unramified setting.\\
$(b)$ Assuming Conjecture \ref{R}, we deduce that if $(K,v)$ is a henselian valued field extending $(\F_p(t),v_t)$ with $k=\F_p$ and $(\Gamma,vt)\equiv_{\exists} (\Z,1)$ in $L_{\text{oag}}$ with a constant for the value of $t$, then $K\equiv_{\exists} \F_p(\!(t)\!)$ in $L_{t}$. It is worth noting that the condition $(\Gamma,vt)\equiv_{\exists} (\Z,1)$ is simply equivalent to asking that $vt$ be minimal positive in $\Gamma$ (cf. Corollary 1.6 \cite{weisp}).
\end{rem}
\begin{rem}
If we do not ask for $(K,v)$ and $(L,w)$ to be tamely ramified over $(\Q,v_p)$ (resp. $(\F_p(t),v_t)$), the conclusion of Corollary \ref{cor1} may fail (see Example \ref{abhexample}). 
\end{rem}

\bc \label{dec1}
Assume Conjecture \ref{R}. Suppose $(K,v)$ is henselian and tamely ramified over $(\Q,v_p)$ (resp. $(\F_p(t),v_t)$ or $(\Q(t),v_t)$). 
Suppose $(K,v)$ admits a cross-section that extends a cross-section of $(\Q,v_p)$ (resp. $(\F_p(t),v_t)$ or $(\Q(t),v_t)$). 
Then $K$ is existentially decidable in $L_{\text{rings}}$ (resp. $L_t$) relative to $k$ in $L_{\text{rings}}$ and $(\Gamma,vp)$ (resp. $(\Gamma,vt)$) 
in $L_{\text{oag}}$ with a constant for the value of $p$ (resp. $t$).
\ec
\begin{proof}
We again focus on the mixed characteristic version. Assume $k$ (resp. $\Gamma$) is existentially decidable in $L_{\text{rings}}$ (resp. $L_{\text{oag}}$ with a constant for $vp$). Consider the $L_{\text{AKE}}$-theory 
$$T=\text{Hen} \cup T_{\text{res}}\cup T_{\text{vg}}\cup \text{Diag}_K(\Q)$$ 
where $\text{Hen}$ is the theory of henselian valued fields with a cross-section and  
$$T_{\text{res}}=\{ \phi \in L_{\text{rings}}: \phi \mbox{ existential and }k\models \phi\} \cup \{  \phi \in L_{\text{rings}}: \phi \mbox{ universal and }k\models  \phi\}$$
and similarly
$$T_{\text{vg}}=\{ \phi \in L_{\text{oag},vp}: \phi \mbox{ existential and }\Gamma\models \phi\} \cup \{ \phi \in L_{\text{oag},vp}: \phi \mbox{ universal and }\Gamma\models  \phi\}$$ 
Finally, $\text{Diag}_K(\Q)$ is the atomic diagram of $\Q$ in $K$ in $L_{\text{AKE}}$. By our assumptions on $k$ and $\Gamma$, we have that the above axiomatization is recursive. \\
\textbf{Claim:} For every existential or universal sentence $\phi \in L_{\text{rings}}$, we have $T\models \phi \iff K\models \phi$.
\begin{proof}
If $T\models \phi$, then clearly $K\models \phi$. For the converse, suppose that $K\models \phi$ and let $(L,w,s_{L})\models T$ with residue field $l$, value group $\Delta$ and cross-section $s_L:\Delta\to L^{\times} $. Note that $\Delta/\Z wp$ has no $p$-torsion elements, using that $(\Delta,wp)\equiv_{\exists} (\Gamma,vp)$ and that $\Gamma/ \Z vp$ has no $p$-torsion elements. 
%
Since $(\Q,v_p)$ is defectless and $l/\F_p$ is separable, we deduce that $(L,w)/(\Q,v_p)$ is tamely ramified. Since $L\models \text{Diag}_K(\Q)$, we have that $s_L$ and $s_K$ restrict to the same cross-section of $(\Q,v_p)$. 
By Corollary \ref{cor1}, we see that $L \models \phi$. 
\qedhere $_{\textit{Claim}}$ \end{proof}
In particular, the theory $T$ is \textit{existentially complete} with respect to $L_{\text{rings}}$, meaning that for every existential sentence $\phi \in L_{\text{rings}}$ either $T\models \phi$ or $T\models \lnot \phi$. A brute-force enumeration of all proofs from the axioms of $T$ now yields an effective procedure for deciding whether $K\models \phi$, for any existential sentence $\phi \in L_{\text{rings}}$. 
\end{proof}
\begin{rem} \label{transisessential}
Even if $(K,v)$ is algebraic over $(\Q,v_p)$, the model $(L,w)$ constructed in the proof of Corollary \ref{dec1} will generally be highly transcendental, the residue field $l$ will also be highly transcendental and $\Delta$ will be of rank greater than $1$. It is therefore essential---even if one is merely interested in algebraic extensions---that we have proved Theorem \ref{mainA} in this level of generality.
\end{rem}
Once again, the cross-section condition cannot be omitted from Corollary \ref{dec1}. We provide a counter-example in equal characteristic $0$:
\begin{example} \label{tamex}
For each $\alpha \in 2^{\omega}$, we define an equal characteristic $0$ valued field $K_{\alpha}$ as follows. First, write $\alpha \restriction n$ for the restriction of $\alpha$ to $n=\{0,1,...,n-1\}$ if $n>0$ and $\alpha \restriction 0=0$.  Now define inductively:
\begin{enumerate}
\item $(K_0,v_t)=(\Q_2(\!(t)\!),v_t)$ and $\pi_{0}=t$.

\item $K_{\alpha \restriction n}=K_{\alpha \restriction (n-1)}((2^{\alpha(n-1)}\cdot \pi_{\alpha \restriction (n-1)})^{1/2} )$ and $\pi_{\alpha \restriction n}=(2^{\alpha(n-1)} \cdot \pi_{\alpha \restriction (n-1)})^{1/2} $ for $n>0$.
\end{enumerate}
More succinctly, for $n\in \N$ we have that $K_{\alpha \restriction n}=\Q_2(\!(t)\!) ((2^{\overline{\alpha}_n} \cdot t)^{1/2^n})$, where $\overline{\alpha}_n= \sum_{k=0}^{n-1} \alpha(k)\cdot 2^{k}$.
We let $K_{\alpha}=\bigcup_{n\in \N} K_{\alpha \restriction n}$. For every $\alpha\in 2^{\omega}$, we have that $(K_{\alpha},v_t)$ is henselian, being an algebraic extension of $(\Q_2(\!(t)\!),v_t)$. It has value group $(\Gamma_{\alpha}, vp)=(\frac{1}{2^{\infty}}\Z,1)$ and residue field $k_{\alpha}=\Q_2$, both of which are decidable. Indeed, the former is decidable using results by Robinson-Zakon \cite{Rob} and the latter is decidable by Ax-Kochen/Ershov.\footnote{The former needs some explanation: Let $T$ be the theory in $L_{\text{oag}}$ that requires of $\Gamma$ that it is regularly dense, that $[\Gamma:2\Gamma]=1$ and $[\Gamma:p\Gamma]=p$ for any prime $p>2$. By Theorem 4.4 \cite{Rob} (see also the proof), $T$ is model-complete in $L_{\text{oag}}\cup \{P_n:n\in \N \}$ (where $P_n(x)\leftrightarrow \exists y (ny=x)$) and also complete. Consider the expansion $L_{\text{oag}}\cup \{P_n:n\in \N \} \cup \{1\}$ and the theory $T'$, which in addition asks that $1>0$ and that $p\nmid 1$ if $p>2$. It is clear that $T'$ is still model-complete and that $(\frac{1}{2^{\infty}}\Z,1)$ is a prime model of $T'$. The decidability of $(\frac{1}{2^{\infty}}\Z,1)$ in $L_{\text{oag}}$ follows. \label{footnote14}} 

If $\alpha \neq \beta$, then $K_{\alpha}\not \equiv_{\exists} K_{\beta}$ in $L_{t}$. Indeed, suppose that $K_{\alpha} \equiv_{\exists} K_{\beta}$ and let $n\in \omega$ be least such that $\alpha(n)\neq \beta(n)$ and say $\alpha(n)=1$. We would then have $a,b\in K_{\alpha}$ such that $a^{2^n}=2^{\overline{\alpha}_n}\cdot t$ and $b^{2^n}= 2^{\overline{\beta}_n} \cdot t$. It follows that $c^{2^n}=2^{2^{n-1}}$, where $c=\frac{a}{b}$. Reducing this equation over the residue field $k_{\alpha}=\Q_2$, we get $\overline{c}^{2^n}=2^{2^{n-1}}\Rightarrow 2\cdot v_2 \overline{c}=1$, which has no solution in the value group $\Z$ of $\Q_2$. We conclude that $K_{\alpha}\not \equiv_{\exists} K_{\beta}$ in $L_{t}$. Since $2^{\omega}$ is uncountable and there are countably many Turing machines, there must exist an $\alpha \in 2^{\omega}$ such that $K_{\alpha}$ is $\exists$-undecidable in $L_{t}$.
\end{example}

\begin{rem}
The algorithm provided by the proof of Corollary \ref{dec1} is effective only in theory, meaning that it is very difficult to implement in practice. A similar remark applies to the proof of Denef-Schoutens \cite{Den} (see Remark \ref{lastrem}$(b)$). In \S \ref{Denef-Schoutensec}, we will present an alternative (conditional) proof of the existential decidability of $\F_p(\!(t)\!)$, which becomes effective once an effective desingularization algorithm in positive characteristic becomes available (see Remark \ref{lastrem}$(b)$).
\end{rem}


\subsubsection{Proof of Corollary \ref{maincor}}
Among the fields that are existentially decidable, the maximal tamely ramified extensions of $\Q_p$ and $\F_p(\!(t)\!)$ occur naturally in ramification theory and are of arithmetic significance. 
\begin{Corol}[Ramification fields]
Assume Conjecture \ref{R}. Then the field $\Q_p^{tr}$ (resp. $\F_p(\!(t)\!)^{tr}$) is existentially decidable in $L_{\text{rings}}$ (resp. $L_t$).
\end{Corol} 
\begin{proof}
From Fact \ref{expl}, one sees that both of these fields have residue field $\overline{ \F}_p$ and value group $\Z_{(p)}$. The field $\overline{ \F}_p$ is decidable in $L_{\text{rings}}$ and $(\Z_{(p)},1)$ is decidable in $L_{\text{oag}}$ with a constant symbol for $1$. The latter is an application of Robinson-Zakon \cite{Rob} (cf. footnote \ref{footnote14}). Moreover, the field $\Q_p^{tr}$ (resp. $\F_p(\!(t)\!)^{tr}$) admits a cross-section mapping $\gamma\mapsto p^{\gamma}$ (resp. $\gamma\mapsto t^{\gamma}$). The conclusion follows from Corollary \ref{dec1}.
\end{proof}
Other examples include $\Q_p(p^{1/\ell^{\infty}})$ and $\F_p(\!(t)\!)(t^{1/\ell^{\infty}})$, where $\ell$ is a prime different from $p$. These are again existentially decidable in $L_{\text{rings}}$ and $L_t$ respectively.
\subsection{Taming Abhyankar's example} \label{tweakabh}
For an application of a different kind, we present a tame variant of the following famous example due to Abhyankar \cite{Abh2}. It is also presented by Kuhlmann in Example 3.13 \cite{defect} in relation to the defect:
\begin{example} \label{abhexample}
Let $(K,v)=(\F_p(\!(t)\!)^{1/p^{\infty}},v_t)$ and $(L,w)=(\F_p(\!(t^{\Gamma_p})\!),v_t)$ be the Hahn series field with value group $\Gamma_p=\frac{1}{p^{\infty}}\Z$ and residue field $\F_p$. We observe that $\text{RV}(K)\cong_{\text{RV}(\F_p(\!(t)\!) )} \text{RV}(L)$ but $(K,v) \not \equiv_{\exists, \F_p(\!(t)\!)} (L,w)$ since the Artin-Schreier equation $x^p-x-1/t=0$ has a solution in $L$ but not in $K$. Note that both $K,L$ admit a cross-section which sends $\gamma\mapsto t^{\gamma}$. This example therefore demonstrates why Corollary \ref{cor1} is not true without the tameness assumption.
\end{example}
Our version of Abhyankar's example is obtained by replacing $p$-power roots of $t$ with $\ell$-power roots and exhibits a totally different behaviour (as expected):
\begin{example} \label{lexample}
Fix any prime $\ell\neq p$. Consider the valued fields $(K,v)=(\F_p(\!(t)\!)(t^{1/\ell^{\infty}}),v_t)$ and $(L,w)=(\F_p(\!(t^{\Gamma_{\ell}})\!),v_t)$, with the latter being the Hahn series field with value group $\Gamma_{\ell}=\frac{1}{\ell^{\infty}}\Z$ and residue field $\F_p$. We observe that $\text{RV}(K) \cong \text{RV}(L)$ and by Theorem \ref{mainA} we get that $(K,v)\equiv_{\exists, \F_p(\!(t^{1/\ell^n})\!)} (L,w)$, for all $n\in \N$. It follows that $K \preceq_1 L$ in $L_{\text{rings}}$.
\end{example}
\begin{rem}
Similarly, we have that $\Q_p(p^{1/\ell^{\infty}})$ is existentially closed in every maximal immediate extension. 
\end{rem}

\section{Revisiting Denef-Schoutens} \label{Denef-Schoutensec}
\subsection{Overview of Denef-Schoutens}
In Theorem 4.3 \cite{Den}, Denef-Schoutens proved the existential decidability of $\F_p(\!(t)\!)$ in $L_t$, assuming resolution of singularities for schemes over fields (Conjecture 1 \cite{Den}). We state Conjecture 1 \cite{Den} below for the convenience of the reader: 
\begin{Conje} \label{conjecture1}
Let $X$ be a reduced scheme of finite type over a field $k$. Then there exists a blow-up morphism $f:\tilde{X}\to X$ in a nowhere dense center $Z\subsetneq X$ such that $\tilde{X}$ is regular.
\end{Conje} 
Using Conjecture \ref{conjecture1}, Denef-Schoutens reduce the problem of existential decidability of $\F_p(\!(t)\!)$ in $L_t$ to the problem of deciding whether a given scheme $Y$ of finite type over $\F_p[\![t]\!]$ has an $\F_p[\![t]\!]$-integral point (see the proof of Theorem 4.3 \cite{Den}). They solve the latter using an \textit{effective} Greenberg approximation theorem (Theorems 3.2, 6.1 \cite{vddbeck}). In particular, they prove \textit{unconditionally} that the \textit{positive} existential theory of $\F_p[\![t]\!]$ is decidable in $L_t$ (Proposition 3.5 \cite{Den}).
\subsection{Comparison with Denef-Schoutens}
Assuming Conjecture \ref{R}, our Theorem \ref{mainA} also implies the existential decidability of $\F_p(\!(t)\!)$ in $L_t$ (see Corollary \ref{dec1}). Moreover, by Remark \ref{crossrem}$(b)$ we have a simple system of axioms which captures the existential theory of $\F_p(\!(t)\!)$ in $L_t$. Note that our proof does not make use of Greenberg's approximation theorem. On the other hand, Conjecture \ref{R} is more refined than Conjecture \ref{conjecture1}. We shall now provide a simplified proof in the case of $\F_p(\!(t)\!)$, which only relies on Conjecture \ref{conjecture1}. 

\subsection{Simplified proof for $\F_p(\!(t)\!)$}
We refer to Appendix \ref{appalg} for background material related to computational algebraic geometry. Let $X_0$ be a \textit{given} affine scheme of finite type over $\F_p[t]$ 
and $W$ a given Zariski open subset of the generic fiber $(X_0)_{\F_p(t)}$. Our task is to decide whether there exists $P\in X_0(\F_p[\![t]\!])$ such that $P_{\eta}\in W(\F_p(\!(t)\!) )$, where $P_{\eta}$ is the underlying $\F_p(\!(t)\!)$-rational point. We note that $X$ and $W$ may be viewed as algorithmic inputs in the sense of \ref{schemesasinputs}.

We shall also write $X=X_0\times_{\Spec(\F_p[t])} \Spec(\F_p[\![t]\!])$ for the base change via $\Spec(\F_p[\![t]\!])\to  \Spec(\F_p[t])$ and $X_s$ for the special fiber.

\subsubsection{Non-singular case} \label{nonsingcase}
Our proof relies on the observation that one can avoid using the effective Greenberg approximation theorem to check if $X$ has an $\F_p[\![t]\!]$-integral point, in case $X$ is \textit{regular} at all points $x\in X_s$. 
Instead, one can use a more elementary fact. Since we were not able to find a reference for Proposition \ref{regsmooth} in published literature, we shall spell out the details. We do mention however that Proposition \ref{regsmooth} follows easily from Proposition 2, pg.61 \cite{Bosch}. 
\bp \label{regsmooth} 
Let $R$ be a DVR with residue field $\kappa$ and $f : X \to \Spec R$ be a morphism of finite type. Let $P$ be an $R$-integral point of $X$, meeting the special fiber at $x\in X_s(\kappa)$ and suppose that $X$ is regular at $x$. Then $f$ is smooth at $x$. 
\ep 
\begin{proof}
By the definition of a smooth morphism (Definition 4.3.35 \cite{Liu}), it suffices to show that $f:X\to \Spec R$ is flat at $x$ and that $X_s$ is smooth at $x$, as an algebraic variety over $\kappa$. To this end, we may assume that $X$ is the local scheme $\Spec(A)$, where $A$ is a regular local ring (by replacing $X$ with $\Spec(\Oo_{X,x})$).\\
\textbf{Claim 1:} The morphism $f$ is flat at $x$.
\begin{proof}
Since $A$ is a regular local ring, it is also an integral domain (Proposition 4.2.11 \cite{Liu}). Since $X$ admits an $R$-integral point, we get that $f:R\to A$ is injective. We conclude that $A$ is a torsion-free $R$-module and therefore flat (see Tag 0539 \cite{sp}). 
\qedhere $_{\textit{Claim 1}}$ \end{proof}

An integral point $P:\Spec R\to X$ corresponds to a section of $f:X\to \Spec R$. The maps $f$ and $P$ induce ring homomorphisms $f^*:R/\mathfrak{m}_R\to A/\mathfrak{m}_A$ and $P^*:A/\mathfrak{m}_A\to R/\mathfrak{m}_R$ such that $P^*\circ f^*=id$. The latter condition means that $P^*$ is surjective. On the other hand, a ring homomorphism between fields is always injective, whence $P^*:A/\mathfrak{m}_A\xrightarrow{\cong} R/\mathfrak{m}_R=\kappa$ is an isomorphism. Since $R$ is a DVR, we may consider $\mathfrak{m}_R/\mathfrak{m}_R^2$ as a $1$-dimensional $\kappa$-vector subspace of $\mathfrak{m}_A/\mathfrak{m}_A^2$, spanned by $ \pi+\mathfrak{m}_A^2$, where $\pi$ is a uniformizer of $R$. One may extend $\{\pi+\mathfrak{m}_A^2 \}$ to a basis of $\mathfrak{m}_A/\mathfrak{m}_A^2$, say $\{\pi+\mathfrak{m}_A^2,x_1+\mathfrak{m}_A^2,...,x_n+\mathfrak{m}_A^2\}$ with $x_i\in \mathfrak{m}_A$. \\
\textbf{Claim 2:} The scheme $X_s$ is smooth at $x$.
\begin{proof}
Let $B=A\otimes_R R/\mathfrak{m}_R\cong A/\pi A$ be the local ring of the special fiber at $x$. We will then have that $\{x_1+\mathfrak{m}_B^2,...,x_n+\mathfrak{m}_B^2\}$ is a $\kappa$-basis for $\mathfrak{m}_B/\mathfrak{m}_B^2$. By flatness, we get that $\dim(B)=\dim(A)-1$ (Theorem 4.3.12 \cite{Liu}) and therefore that the local ring $B$ is regular, i.e. $X_s$ is regular at $x$. Since $\kappa(x)=\kappa$, we get that $X_s$ is smooth at $x$ (Proposition 4.3.30 \cite{Liu}).
\qedhere $_{\textit{Claim 2}}$ \end{proof}

\end{proof}
We refer to Definition \ref{givenschemeconve}, which allows us to view affine schemes of finite type (over a computable base ring) as algorithmic inputs.
\bc \label{corint}
There exists an algorithm to decide whether a given regular affine scheme $X_0$ of finite type over $\F_p[t]$ has an $\F_p[\![t]\!]$-integral point.
\ec
\begin{proof}
We keep our notation from \S \ref{nonsingcase}; $X=X_0\times_{\Spec(\F_p[t])} \Spec(\F_p[\![t]\!])$ and $X_s$ is the special fiber of $X_0$, which is also identified with the special fiber of $X$. Say $X_0=\Spec(A)$ where $A=\F_p[t,x_1,...,x_m]/(f_1,...,f_s)$ with $f_i\in \F_p[t,x_1,...,x_m]$. Any $\F_p[\![t]\!]$-integral point of $X_0$ corresponds to an $\F_p[\![t]\!]$-integral point of one of its irreducible components. Since we may effectively compute the irreducible components of $X_0$ (see Fact \ref{seidenfact}), this allows us to further assume that $X_0$ is irreducible. Since $X_0$ is already reduced (being regular), we have that $X_0$ is integral.

For any $x\in X_s$, we have that $X_0$ is regular at $x$ if and only if $X$ is regular at $x$ (see Tag 0BG6$(2)$ \cite{sp}). Together with our assumption, this gives that $X$ is regular at all points $x\in X_s$. Now $X_s$ has finitely many $\F_p$-rational points, the set of which is computable by brute-force. It is therefore enough to check whether a given $x\in X_s(\F_p)$ lifts to an $\F_p[\![t]\!]$-integral point. By Proposition \ref{regsmooth} and Hensel's Lemma (Corollary 6.2.13 \cite{Liu}), we equivalently need to check if $X\to \Spec(\F_p[\![t]\!])$ is smooth at $x$. \\
\textbf{Claim 1:} We can effectively check if $X\to \Spec(\F_p[\![t]\!])$ is flat at $x$.
\begin{proof}
Equivalently, we need 
to see if $t$ is a zero-divisor in $A$. 
Since $A$ is an integral domain, we need to check if $t=0$ in $A$ (i.e., if $X_0=X_s$). Equivalently, we need to check if the generic fiber $(X_0)_{\F_p(t)}=\Spec(\F_p(t)[x_1,...,x_m]/(f_1,...,f_s))$ is non-empty. Equivalently, we need to see if $(f_1,...,f_s)=(1)$ in $\F_p(t)[x_1,...,x_m]/(f_1,...,f_s)$, which can be checked effectively by Fact \ref{seidenfact2}.
\qedhere $_{\textit{Claim 1}}$ \end{proof}

Note that $X_s=\Spec(\F_p[x_1,...,x_m]/(\overline{f}_1,...,\overline{f}_s)  )$, where $\overline{f}_i(x_1,...,x_m)=f_i(0,x_1,...,x_m)$. Finally, we have the following:\\
\textbf{Claim 2:} We can effectively check whether $X_s$ is smooth at $x$.
\begin{proof}
This can be done by using the Jacobian criterion for smoothness (Theorem 4.2.19 \cite{Liu}). More precisely, one needs to check if $\text{rank}(J_x)=m- \dim_x X_s$, where $J_x:=(\frac{\partial \overline{f}_i}{\partial x_j}(x))_{1\leq i \leq s, 1 \leq j \leq m}$. To this end, we note that:
\begin{enumerate}
\item The number $\text{rank}(J_x)$ is computable by Gauss elimination.

\item The number $\dim_x X_s$ is also computable. One needs to compute the irreducible decomposition $X_s=\bigcup_{i=1}^n X_i$ using Fact \ref{seidenfact} and calculate $\max_{x\in X_i} \dim X_i$ using Fact \ref{seidenfact'}.
\end{enumerate} 
\qedhere $_{\textit{Claim 2}}$ \end{proof}
 Claim 2 finishes the proof.
\end{proof}
\subsubsection{General case} \label{gencase}
We now sketch the case where $X_0$ is a general affine scheme over $\F_p[t]$. This follows again by induction on $\dim(X_0)$. If $\dim(X_0)=0$, then $X_0$ cannot possibly have any $\F_p[\![t]\!]$-integral points. For the inductive step, we first view $X_0$ as a scheme of finite type over $\F_p$ via $X_0 \to \Spec(\F_p[t])\to \Spec(\F_p)$. As in the proof of Corollary \ref{corint}, we may assume that $X_0$ is reduced, by computing \textit{effectively} $(X_0)_{\text{red}}$ (see Fact \ref{seidenfact} and Remark \ref{seidenrem}) and replacing $X_0$ with $(X_0)_{\text{red}}$. 
We may therefore take $X_0$ to reduced and $W$ a \textit{non-empty} Zariski open subset of the generic fiber $(X_0)_{\F_p(t)}$. We note that the case $W=\emptyset$ can again be checked effectively: If $W^c=V(J)$ and $(X_0)_{\F_p(t)}=V(I)$, it suffices to calculate $\sqrt{J}$ (Fact \ref{seidenfact}) and check if $I=J$ (see Fact \ref{seidenfact2}). In that case, the output of our algorithm is that there does not exist $P\in X_0(\F_p[\![t]\!])$ with $P_{\eta}\in W(\F_p(\!(t)\!) )$. 

Now Conjecture \ref{conjecture1} provides us with a blow-up morphism $Y_0\to X_0$ in a nowhere dense center $Z\subsetneq X_0$ with $Y_0$ regular. The center of such a blow-up can be calculated by brute-force, as explained in Remark 4.1 \cite{Den}. The affine charts of $Y_0$ can be computed explicitly from the blow-up ideal. Using Corollary \ref{corint} for each affine piece separately, one may then check if $Y=Y_0\times_{\Spec(\F_p[t])} \Spec(\F_p[\![t]\!])$ has an $\F_p[\![t]\!]$-integral point. Moreover, if $Y$ has an $\F_p[\![t]\!]$-integral point, its underlying rational point is also a regular point of $Y$, using that the regular locus $\text{Reg}(Y)$ is open in $Y$. Arguing as in the proof of Theorem 4.3 \cite{Den}, we see that there are two scenarios: 
\begin{enumerate}
\item Either $Y(\F_p[\![t]\!])\neq \emptyset$, which implies the existence of $P\in X(\F_p[\![t]\!])$ with $P_{\eta} \in W(\F_p(\!(t)\!) )$ (using Theorem 2.4 \cite{Den}) or

\item The problem is reduced to the lower-dimensional scheme $Z$, which can be solved by our induction hypothesis.
\end{enumerate}

\begin{rem} \label{comparwithds}
$(a)$ In Theorem 4.3 \cite{Den}, Denef-Schoutens desingularize the \textit{generic} fiber of $X$. As a consequence, the scheme $Y$ produced in the proof of Theorem 4.3 \cite{Den} has regular generic fiber but $Y$ need not be regular at points $y\in Y_s$. On the other hand, by desingularizing $X_0$, our scheme $Y_0$ (and hence $Y$) is regular at all $y\in Y_s$. This simplifies the task of deciding whether $Y(\F_p[\![t]\!])=\emptyset$, because of Corollary \ref{corint}.\\ 
$(b)$ The state of the art in desingularization theory allows the present method to work unconditionally for $\dim (X_0)=3$, whereas the method of \cite{Den} works for $\dim (X_0)=4$. This is an unfortunate consequence of desingularizing $X_0$ instead of the generic fiber of $X$.
\end{rem}

\begin{rem} \label{lastrem}
$(a)$ One defect of this method, besides the one mentioned in Remark \ref{comparwithds}$(b)$, is that it does not provide us with an \textit{unconditional} proof of the decidability of the \textit{positive} existential theory of $\F_p[\![t]\!]$ in $L_t$. \\
$(b)$ One advantage of this method is that once an \textit{effective} desingularization algorithm is known in characteristic $p$, the above proof can be converted into an actually effective (not merely theoretically terminating) algorithm. On the other hand, the effective Greenberg approximation theorem used by Denef-Schoutens (see Theorem 3.1, Remark 3.3 \cite{Den}) is only effective in theory, as it ultimately relies on the brute-force algorithm explained in the proof of Theorem 6.1 \cite{vddbeck}. 
\end{rem}

\appendix
\section{Computational algebraic geometry} \label{appalg}
\subsection{Affine schemes as algorithmic inputs} \label{app1} \label{schemesasinputs}
\begin{definition}
Let $R$ be a Noetherian and \textit{computable} ring, i.e. one whose underlying set is (or may be indetified with) a recursive subset of $\N$, so that the ring operations are (or are identified with) recursive functions (e.g., $R=\F_p[t]$ or $R=\F_p(t)$ or $R=\F_p$). 
\end{definition}
Ideals $I\subseteq R[x_1,...,x_n]$ are construed as algorithmic inputs via some natural identification $R[x_1,...,x_n]\simeq \N$. For example, one can enumerate all finite sequences of elements from $R[x_1,...,x_n]$, which encode (non-faithfully) all ideals of $R[x_1,...,x_n]$ by Noetherianity.
\begin{definition} \label{givenschemeconve}
$(a)$  By saying that we are \textit{explicitly given} (or simply given) an ideal $I\subseteq R[x_1,...,x_n]$, we mean that we are given a finite set of generators via the above identification.\\
$(b)$ By saying that we are \textit{explicitly given} (or simply given) a closed subscheme $X\subseteq \mathbb{A}_R^n$, we mean that $X=\Spec(A)$ with $A=R[x_1,...,x_n]/I(X)$ and the ideal $I(X)\subseteq R[x_1,...,x_n]$ is given. 
\end{definition}

%
\subsection{Some algorithms in algebraic geometry} \label{app2}
We recall some standard facts from computational algebraic geometry. We stress that in order to effectively compute primary decomposition of ideals (and the computation of the associated primes), the choice of the base field $k$ is delicate. Merely assuming that $k$ is computable is not enough. This is explained in detail in the introduction of \cite{Seidencon}. One needs to stipulate the following two conditions below:
\begin{definition}
Let $k$ be a computable field.\\
$(a)$ We say that $k$ satisfies condition (F) if one can effectively compute the factorization of any given polynomial in $k[x]$. \\
$(b)$ We say that $k$ satisfies condition (P) if for any given linear homogeneous system $\sum_{1\leq i,j\leq n} a_{ij} X_j=0$, one can effectively check if there exists a solution in $k^p$, and if it does to find one.
\end{definition}
\begin{fact}  \label{seidenfact}
Let $k$ be a computable field satisfying conditions (F) and (P). There is an algorithm to compute the primary decomposition of any given ideal $I\subseteq  k[x_1,...,x_n]$ as well as the associated primes.
\end{fact}
\begin{proof}
See 36, pg. 290 \cite{Seidencon} and 42, pg. 291 \cite{Seidencon}. 
\end{proof}
In \S \ref{Denef-Schoutensec}, we apply Fact \ref{seidenfact} only in the cases of $k=\F_p$ and $k=\F_p(t)$. We note that $\F_p$ clearly satisfies (F) and (P), while for $\F_p(t)$ it follows from 41 \cite{Seidencon}.
\begin{rem} \label{seidenrem}
By the usual dictionary between algebra and geometry, Fact \ref{seidenfact} says that we can compute the irreducible components of a given closed subscheme $X\subseteq \mathbb{A}_{k}^n$. Being able to compute the associated primes of the primary decomposition of a given ideal $I\subseteq  k[x_1,...,x_n]$, in particular allows us to compute the reduced underlying scheme of any given closed subscheme $X\subseteq \mathbb{A}_{k}^n$. 
\end{rem}
The next two facts are standard applications of Gr\"obner bases (see \cite{cox}).
\begin{fact} \label{seidenfact'}
Let $k$ be any computable field. There is an algorithm to compute the dimension of any given affine variety $V\subseteq \mathbb{A}_k^n$
\end{fact}  
\begin{proof}
One may perform computations over $\overline{k}$. Now see Chapter 9 \S 3 \cite{cox}, especially Theorem 8 and the discussion after that.
\end{proof}

\begin{fact} \label{seidenfact2}
Let $k$ be a computable field. There is an algorithm to check if $f\in I$, for any given $f\in k[x_1,...,x_n]$ and $I\subseteq k[x_1,...,x_n]$. In particular, there is an algorithm to check if $I=J$ for any two given ideals $I,J \subseteq  k[x_1,...,x_n]$.
\end{fact}
\begin{proof}
One may check this over $\overline{k}$ and use Chapter 2, \S 8 \cite{cox}.
\end{proof}
%


\section*{Acknowledgements}
I am indebted to the anonymous referee for several comments and corrections which have drastically improved the quality of this paper. I would like to thank E. Hrushovski and J. Koenigsmann for invaluable guidance. I also thank M. Temkin for answering my questions related to resolution of singularities and F.-V. Kuhlmann for instructive feedback.
\bibliographystyle{alpha}
\bibliography{references3}
\Addresses
\end{document}